\documentclass[a4paper, 11pt]{amsart}

\usepackage{enumerate}
\usepackage{mathrsfs}
\usepackage{amssymb}
\usepackage[all]{xy}

\title[]
{Unstable loci in flag varieties\\ and\\ variation of quotients}
\author{Henrik Sepp\"anen, Valdemar V. Tsanov} 

\thanks{V.V.T. is supported by the DFG grant Sachbeihilfe DFG-AZ: TS 352/1-1.}

\date{\today}

\newcommand{\C}{\mathbb{C}}
\newcommand{\R}{\mathbb{R}}
\newcommand{\N}{\mathbb{N}}
\newcommand{\Z}{\mathbb{Z}}
\newcommand{\Q}{\mathbb{Q}}
\newcommand{\mP}{\mathbb{P}}

\def\Lw  {\Longrightarrow}

\def\mc {\mathcal}
\def\mk {\mathfrak}

\def\ol  {\overline}

\def\tst {\Longleftrightarrow}

\def\wt  {\widetilde}

\newtheorem{prop}{Proposition}[section]
\newtheorem{lemma}[prop]{Lemma}
\newtheorem{cor}[prop]{Corollary}
\newtheorem{thm}[prop]{Theorem}
\theoremstyle{definition}
\newtheorem{defin}[prop]{Definition}
\newtheorem{rem}[prop]{Remark}
\newtheorem{example}{Example}

\begin{document}

\maketitle

\vspace{-0.5cm}

\begin{abstract}
We consider the action of a semisimple subgroup $\hat G$ of a semisimple complex group $G$
on the flag variety $X=G/B$, and the linearizations of this action by line bundles $\mc L$ on $X$. 
The main result is an explicit description of the associated {\it unstable locus} in dependence of $\mc L$,
as well as a combinatorial formula for its (co)dimension. We observe that the codimension is equal to 1 on the regular boundary of the $\hat G$-ample cone, 
and grows towards the interior in steps by 1, in a way that the line bundles with unstable locus of codimension $q$ form a convex polyhedral cone. We also give a recursive algorithm for determining all GIT-classes in the $\hat G$-ample cone of $X$.

As an application, we give conditions ensuring the existence of GIT-classes $C$ with
an unstable locus of codimension at least two and which moreover yield geometric GIT quotients.
Such quotients $Y_C$ reflect global information on $\hat G$-invariants. They are always Mori dream spaces, and the Mori
chambers of the pseudoeffective cone $\overline{{\rm Eff}}(Y_C)$ correspond to
the GIT-chambers of the $\hat G$-ample cone of $X$. Moreover, all rational contractions
$f: Y_{C} \dasharrow Y'$ to normal projective varieties $Y'$ are induced by GIT from
linearizations of the action of $\hat G$ on $X$.
In particular, this is shown to hold for a diagonal embedding $\hat G \hookrightarrow (\hat G)^k$,
with sufficiently large $k$. 
\end{abstract}

\vspace{-0.3cm}

\small{
\tableofcontents
}


\section{Introduction}

One of the fundamental problems in representation theory, occurring in various situations, is the understanding of the space of invariants
$V^{\hat G}$, where $\hat G\to G$ is a morphism of groups and $V$ is a representation space of $G$. We apply the framework of Variation of Geometric Invariant Theory, VGIT (in the sense of 
\cite{Dolga-Hu},\cite{Thaddeus-GITflips}), to embeddings of semisimple complex algebraic groups
$\iota:\hat G\subset G$ and study comparatively the $\hat G$-action on the complete flag variety $X=G/B$, where $B$ is a Borel subgroup, on one hand,
and the variations in the dimension of the space of invariants $V^{\hat G}$ for finite dimensional irreducible $G$-modules $V$, on the other hand.
The setting is classical, developed to a large extent in works related originally to the Horn problem formulated in terms of decompositions of tensor products
of representations, here $GL_n=\hat G \stackrel{diag}{\hookrightarrow}G=GL_n^{\times k}$, see \cite{Belk-Kumar-Ress} for a modern account. The general branching problem
of decomposing $G$-modules over $\hat G$ can be formulated as the problem of invariants for the diagonal embedding ${\rm id}\times\iota:\hat G\subset \hat G\times G$,
via the isomorphism ${\rm Hom}_{\hat G}(\hat V,V)\cong (\hat V^*\otimes V)^{\hat G}$. Recalling that the irreducible $G$-modules are parametrized up to isomorphism by
the elements of a Weyl chamber $\Lambda^+$ in the character lattice $\Lambda$ of a maximal torus $T\subset B$, one is led to consider the so-called (generalized)
Littlewood-Richardson monoid and cone, and their sections corresponding to invariants, also called the eigenmonoid and eigencone:
\begin{gather*}
\begin{array}{l}
LR(\hat G\subset G) = \{(\hat\lambda,\lambda)\in\hat\Lambda^+\times\Lambda^+ : (\hat V_{\hat\lambda}^*\otimes V_\lambda)^{\hat G}\ne 0\} \;, \\
LR_0(\hat G\subset G) = \{\lambda\in\Lambda^+ : V_\lambda^{\hat G}\ne 0\} \;, \\
\mc{LR}_0(\hat G\subset G) = {\rm Span}_{\R_+}LR_0 \subset \Lambda_\R \;\;,\;\; \mc{LR}(\hat G\subset G) = {\rm Span}_{\R_+}LR \subset \hat\Lambda_\R\times \Lambda_\R \;. 
\end{array}
\end{gather*}
They have been shown by Brion and Knop to be indeed finitely generated monoids and rational polyhedral cones, respectively. Our results concern finer geometric and combinatorial structures in $\mc{LR}_0$ and the global behaviour of the multiplicities $\dim V_{\lambda}^{\hat G}$. The methods are based on the identification of $\mc{LR}_0$ with the $\hat G$-ample cone on the flag variety, $C^{\hat G}(X)\subset{\rm Pic}(X)_\R$, the closed cone generated by the ample line bundles whose section rings admit nontrivial $\hat G$-invariant elements. We study the GIT-classes defined by equality of the unstable locus - the vanishing locus of the nontrivial homogeneous invariants. The identification $\mc{LR}_0\cong C^{\hat G}(X)$ depends on the hypothesis for $\mc{LR}_0$ to contain regular elements, or equivalently, for $C^{\hat G}(X)$ to be nonempty.
This condition fails only in distinguished special cases and, applied to Levi subgroups, turns out to play a key role in the description of the Kirwan-Ness stratification of the unstable loci and the calculation of their dimension. As an outcome, we exhibit a family of Mori dream spaces obtained as GIT-quotients, containing global information on the invariants, and presenting a potential interest for Mori theory and its interactions with the structure theory of semisimple groups.

The GIT approach to the problem of finding invariants for a given $\iota$ has been developed and applied successfully in a series of works, notably by Heckman, Berenstein, Sjamaar, Belkale, Kumar, Ressayre, Richmond and others (see \cite{Heckman-1982}, \cite{Beren-Sjam}, \cite{Belk-Kumar}, \cite{Ressayre-2010-GITandEigen}, \cite{Ress-Rich} and the references therein), culminating in a description of $C^{\hat G}(X)$ by a minimal set of inequalities. Without citing full statements, we sketch some aspects of the solution, since this will help introduce the setting and put our work into context.
The first milestone is the Borel-Weil theorem, providing models for the irreducible $G$-modules as the spaces of sections of effective line bundles on the flag variety, $H^0(X,\mc L_\lambda)=V_\lambda^*$ for $\lambda\in\Lambda^+$, with the isomorphism of lattices ${\rm Pic}(X)\cong \Lambda$, given by $\mc L_\lambda=G\times_B\C_{-\lambda}$; the dominant Weyl chamber $\Lambda^+$ spans the pseudoeffective cone, while the ample line bundles are given by strictly dominant weights $\Lambda^{++}$.
The $\hat G$-ample cone in ${\rm Pic}(X)_\R$ is then given by the line bundles admitting nonconstant invariants in their section rings. This construction fits into the framework of GIT, \cite{Mumfordetal-GIT}, thus providing tools, in particular the Hilbert-Mumford criterion, for the study of the $\hat G$-action and invariants. The inequalities defining $C^{\hat G}(X)$ in $\Lambda_\R$, given in \cite{Beren-Sjam} and optimized in later works, are derived from the Hilbert-Mumford criterion, and have the form
\begin{gather}\label{For Inequality}
\lambda(w\xi) \leq 0 \;,
\end{gather}
where $\lambda$ is the dominant weight representing the line bundle, $w$ is an element of the Weyl group $W$ of $G$, and $\xi$ is an element in the lattice of one-parameter subgroups in a Cartan subgroup of $\hat G$ of the form $\hat T=\hat G\cap T$, identified with the integral coweight lattice $\hat\Gamma\subset\Gamma=\Lambda^{\vee}\subset\mk t$. The pairs $(\xi,w)\in \hat\Gamma\times W$ appearing in the inequalities are subject to certain conditions, and the description of these conditions presents the main technical issue. A finite but redundant list of pairs is given by Berenstein and Sjamaar, minimized by Belkale and Kumar in the diagonal case, and by Ressayre for arbitrary embeddings of reductive groups. The list of relevant elements $\xi$ is relatively easy to obtain, they are determined by the weights of the $\hat G$-action on the quotient of Lie algebras $\mk g/\hat{\mk g}$; in the diagonal case, one has simply the fundamental coweights of $\hat G$.
The relevant Weyl group elements present a more delicate problem. The conditions on $w$, given in the aforementioned series of works starting with \cite{Beren-Sjam}, are cohomological, stated in terms of pullbacks of Schubert classes from flag varieties of $G$ 
to closed $\hat G$-orbits in them. There is an interest in a cohomology-free description of the $\hat G$-ample cone,
and this has been achieved for diagonal embeddings in \cite{Belk-Kumar} with a non optimal list, optimized for some classical groups in term of quiver representations by Derksen-Weyman \cite{Derksen-Weyman-Quiver} for groups of type A, and Ressayre \cite{Ress-CohFree} in types A,B,C.

We obtain a cohomology free description of $C^{\hat G}(X)$ by a finite list of inequalities, redundant in general, as a biproduct of one of our main results. A formulation is given in Theorem \ref{Theo Ck}. The proof is in fact parallel to that of \cite{Beren-Sjam}, the difference being rather formal than essential, but we present a full argument based directly on the Hilbert-Mumford criterion. Our goal is in fact the structure behind the boundary of $C^{\hat G}(X)$.

For a view on the global behaviour of invariants, it is convenient to consider the Cox ring of $X$ (cf. \cite{hk}), or total coordinate ring, consisting in this case of the sum of all irreducible $G$-modules, with its 
subrings generated by the individual line bundles.  
$$
{\rm Cox}(X)=\bigoplus\limits_{\mc L\in{\rm Pic(X)}} H^0(X,\mc L) \cong \bigoplus\limits_{\lambda\in\Lambda^+} V_\lambda \quad, \quad R_\lambda=\bigoplus\limits_{j=0}^\infty H^0(X,\mc L_\lambda^j)=\bigoplus\limits_{j=0}^\infty V_{j\lambda}^* \;.
$$
The $\hat G$-invariants we are after are then all assembled in the invariant ring ${\rm Cox}(X)^{\hat G}$, which is also finitely generated. Cox rings are an important ingredient in the theory of Mori dream spaces, the latter having finitely generated Cox rings as one of their essential defining properties (cf. \cite{hk} for the full definition). The flag varieties form indeed a class of known examples. It is natural to ask about a variety, a quotient $Y$, with ${\rm Cox}(X)^{\hat G}$ as a Cox ring, having the classical result for individual line bundles in mind. Such a variety would be a geometric incarnation of the complete information on invariants for the given pair $\hat G\subset G$. This topic is addressed in \cite{Seppanen-GlobBranch}, where such quotients are constructed and shown to be Mori dream spaces. The construction rests, however, on a nontrivial assumption for existence of $\hat G$-movable chambers in $C^{\hat G}(X)$, defined by line bundles whose 
rings of nontrivial invariants have vanishing locus, the unstable locus $X^{us}(\lambda)$, of codimension at least 2 in $X$, containing all points with positive dimensional stabilizers.

In the first part of this article, we address the question of existence of $\hat G$-movable chambers. We devise a general method addressing this problem built on a closed formula for the unstable locus, formulated in Theorem I in the next section. This formula allows us to study GIT-classes of line bundles and their variations. We give a description of the GIT-classes, showing that all inequalities defining chambers in $C^{\hat G}(X)$ are of same type as (\ref{For Inequality}), and we provide a procedure arriving at the relevant $w\xi$, formulated roughly in Theorem II in the next section, and more precisely in Theorem \ref{Theo hatG-chambersandfacets}. We show that the codimension of the unstable locus is equal to 1 at the regular boundary of the $\hat G$-ample cone grows in steps of 1 towards the interior,
so that there is a sequence of convex polyhedral cones $C^{\hat G}(X)=C_1\supset C_2\supset C_3...$, where $C_j$ spanned by line bundles with codimension of the unstable locus at least $j$. We derive a criterion for existence of $\hat G$-movable chambers in terms of the structure of the embedding $\hat G\subset G$. For diagonal embeddings $\hat G\subset \hat G^{\times k}=G$, we establish the existence of $\hat G$-movable chambers for sufficiently large $k$. 

The second main topic, formulated as Theorem IV in the next section, concerns the birational geometry of a quotient $Y$ by a $\hat G$-movable chamber. We establish a canonical identification of the GIT-chambers in $C^{\hat G}(X)$ with the Mori chambers in the pseudoeffective cone in ${\rm Pic}(Y)$. 

The two theorems are independent of each other. Together they present a method for the detection of the requested quotients and a description of the correspondence between their Cox rings. The content of the article is explained in the next section.

\section{Setting and statement of the main results}

The basic GIT notions relating the geometry of $X=G/B$ to the invariant rings $R_\lambda^{\hat G}$ are the notions of instability, semistability, stability and quotients, \cite{Mumfordetal-GIT}. The central role in the first part of this article is played by the $\hat G$-unstable locus, defined by the vanishing of the nonconstant invariants in the section ring of a given ample line bundle. The ample line bundles on $X$ have the form $\mc L_\lambda$ for $\lambda\in \Lambda^{++}$, and one has
$$
X^{us}(\lambda)=X_{\hat G}^{us}(\lambda)=Z(J_\lambda)\subset X \;,\quad J_\lambda=\bigoplus\limits_{j\geq 1}^\infty H^0(X,\mc L_{\lambda}^j)^{\hat G} = \bigoplus\limits_{j\geq 1}^\infty (V_{j\lambda}^*)^{\hat G}\;.
$$
The semistable locus is the complementary open set $X^{ss}(\lambda)=X_{\hat G}^{ss}(\lambda)=X\setminus X^{us}(\lambda)$. The Hilbert-Mumford criterion (stated in the next section) gives a numerical characterization of the unstable or, equivalently, semistable points, and allows to extend the notions to the $\R$-Picard group, i.e., $\lambda\in\Lambda_\R$. This yields a characterization of the $\hat G$-ample bundles by having an unstable locus of positive codimension or, equivalently, a nonempty semistable locus. In some sense this simple choice represents the main difference between our view on $C^{\hat G}(X)$ and the view we see in the articles mentioned above, since the cohomological conditions for $w$ are obtained from the condition for its Schubert cell to contain semistable points. We focus on instability. 

For the Littlewood-Richardson monoid $LR_0$ we need to consider nonregular dominant weights $\lambda\in\Lambda^+$ as well. These yield semiample line bundles and the above definition of the unstable locus extends to that case. However the structure of quotients may differ. We define the $\hat G$-ample cone on $X$ as the closed cone generated by ample line bundles as
$$
C^{\hat G}(X) = \ol{\{\lambda\in\Lambda_\Q^{++} : \exists q\in\N: J_{q\lambda}\ne 0\}} \subset \Lambda_\R \;.
$$
Whenever $C^{\hat G}(X)$ is nonempty, we have
$$
C^{\hat G}(X) = \mc{LR}_0(\hat G\subset G) = \{\lambda\in\Lambda_\R^+: {\rm codim}_X X^{us}(\lambda)>0\} \subset\Lambda_\R \;.
$$
A typical case where $C^{\hat G}(X)=\emptyset$ is given by the condition that $\hat G$ and $G$ have nontrivial common connected normal subgroups. This condition is in fact necessary and sufficient for the full branching cone $\mc{LR}$ to have full dimension in $\hat\Lambda_\R\times\Lambda_\R$, cf. \cite{Ressayre-2010-GITandEigen}. The cone $\mc{LR}_0$ may have positive codimension in $\Lambda_\R$ and contain, or not contain, regular weights. For instance, for $\hat G=Sp_{2n}\subset SL_{2n}=G$ we have $C^{\hat G}(X)=\emptyset$, while for $\hat G\stackrel{diag}{\hookrightarrow} \hat G\times \hat G=G$, we have $C^{\hat G}(X)=\mc{LR}_0=\{(\hat\lambda,\hat\lambda^*):\hat\lambda\in\hat\Lambda^+\}$.

Our first main result is an explicit description of the Kirwan stratification of the $\hat G$-unstable locus in $X=G/B$ with respect to any ample line bundle, proven in Theorem \ref{Theo KirwanStratFlag}. This yields a closed formula for $X^{us}(\lambda)$, as well as formulae and combinatorial bounds for its dimension. Below we state a simplified version, with a hypothesis on $\lambda$, which allows for a concise statement, and is in fact important in the rest of the discussion. Theorem \ref{Theo KirwanStratFlag} allows arbitrary reductive subgroups and ample line bundles, and contains a precise description of the Kirwan strata and their dimensions. 

Recall that the $T$-fixed points in $X$ are parametrized by the elements of the Weyl group $W$, as $X^T=\{x_w=wB, w\in W\}$; their $B$-orbits give the Schubert cell decomposition $X=\sqcup_w Bx_w$. Suppose, which can be done without loss of generality, that Weyl chambers for $\hat G$ and $G$ are chosen so that $\dim\mk t_+\cap\hat{\mk t}_+=\dim\hat{\mk t}_+$ (real dimension). For any one-parameter subgroup of $T$, $\xi\in\Gamma$, let $P_\xi\subset G$ denote the parabolic subgroup, whose Lie algebra is the sum of the eigenspaces of ${\rm ad}\xi$ with nonnegative eigenvalues. Let $P_1,...,P_q$ be the maximal parabolic subgroups of $G$ among $P_\xi$ with $\xi\in\hat\Gamma^+\setminus\{0\}$. Then there are uniquely determined indivisible dominant one-parameter subgroups $\xi_1,...,\xi_q\in\hat\Gamma^+$ such that $P_j=P_{\xi_j}$. The set $\xi_1,...,\xi_q$ is known to be been related to the Littlewood-Richardson cone, notably in the works of Ressayre, \cite{Ressayre-2010-GITandEigen}. We show how it arises naturally in the study of the Kirwan-Ness stratification. For diagonal embeddings the $\xi_j$'s are just the fundamental coweights of $\hat G$. Set $r_j = \dim G/P_j$, $\hat r_j=\dim \hat G/\hat P_j$.\\

\noindent{\bf Theorem I:} Let $\lambda\in \Lambda^{++}$. Then the $\hat G$-unstable locus can be written as the $\hat G$-saturation of a union of parabolic orbits
\begin{gather*}\label{For Xus Intro}
X_{\hat G}^{us}(\lambda) = \bigcup\limits_{j=1}^q \bigcup\limits_{w\in W: \lambda(w^{-1}\xi_j)>0} \hat GP_{j} x_w
\end{gather*}
Furthermore, denoting $p_j(w)=\dim P_jx_w$, we have
\begin{align*}
& \dim\hat GP_j x_w \leq \hat r_j + p_j(w) \\
& \dim X^{us}(\lambda) = \min\limits_{j,w}\{ \hat r_j + p_j(w) \;:\; \lambda(w^{-1}\xi_j)>0 , \dim \hat GP_\xi x_w=\hat{r_j} + p_j(w)\} \;.
\end{align*}

The closure of a parabolic orbit $\ol{P_j x_w}$ is a Schubert variety, perhaps not for $B$, but for a Borel subgroup $B^j\subset P_j$, Weyl-conjugate of $B$, for which $\xi_j$ is dominant in $G$. The 
dimension and codimension of $\hat GP_\xi x_w$ can be computed in terms of lengths of Weyl group elements, which is very useful for our calculations. The 
condition $\dim \hat GP_jx_w=\hat{r_j}+p_j(w)$ is related to the notions of Levi-movability and the Belkale-Kumar product cohomology of flag varieties, a key notion in the minimal description of the $\hat G$-ample cone in the works of Belkale-Kumar, Ressayre, Richmond, \cite{Belk-Kumar},\cite{Ressayre-2010-GITandEigen},\cite{Ress-Rich}, We develop an independent approach focused on the unstable locus, based directly on the Hilbert-Mumford criterion and the Kirwan-Ness stratification theorem. Our method is rather related to a method used by Popov, \cite{Popov-Nullforms}, to study unstable loci representation spaces. Also, as an addendum independent from the rest of 
our results, we also adopt a combinatorial tree-algorithm from \cite{Popov-Nullforms}, which can be used to determine the set of relevant (stratifying) pairs $(j,w)$. As a 
corollary, we obtain cohomology-free description of $C^{\hat G}(X)$ which has a rather recursive algorithmic character. Our description is not necessarily optimal, redundant inequalities may occur. It is, however, exact and allows us to study the interior of the $\hat G$-ample cone.\\

The $\hat G$-ample cone is subdivided into GIT-equivalence classes, defined by equality of the unstable loci, i.e., $\lambda\sim\lambda'$ if and only if $X^{us}(\lambda)=X^{us}(\lambda')$. For $\hat G$-ample line 
bundles, the projective spectrum of the invariant ring is isomorphic to the GIT-quotient of $X$ defined by Hilbert's equivalence relation on the semistable locus:
$$
Y_\lambda = X^{ss}(\lambda)// \hat G \cong {\rm Proj}(R_\lambda^{\hat G}) \quad,\quad x_1\sim x_2 \Longleftrightarrow \overline{\hat Gx_1}\cap\overline{\hat Gx_2} \cap X^{ss}(\lambda) \ne \emptyset \;.
$$
The quotients defined by GIT-equivalent bundles are clearly isomorphic, and we denote $X^{us}_{\hat G}(C)=X^{us}_{\hat G}(\lambda)$ and $Y_{C}=Y_\lambda$ 
for a GIT-class $C\ni \lambda$. The GIT-classes form a system of cones in $C^{\hat G}(X)$, and there are rational maps between some of these quotients,
depending on relations between the corresponding GIT-classes (cf. \cite{Dolga-Hu}, \cite{Thaddeus-GITflips}, \cite{Ress-GITequiv}).

Some important properties of the quotient are reflected in properties of the unstable locus. Note that the quotient is geometric when the semistable orbits are equidimensional. 
In particular, one considers the set of infinitesimally free orbits closed in the semistable locus, called the stable locus:
$$
X^{s}_{\hat G}(\lambda) = \{ x\in X^{ss}_{\hat G}(\lambda): \hat Gx \subset X^{ss}_{\hat G}(\lambda)\;{\rm closed}, \dim\hat G_{x}=0 \} \;,
$$
where $\hat G_x$ denotes the stabilizer of $x$. The GIT-class of $\lambda$ is called a chamber if all semistable points are stable, i.e., 
$X^{us}_{\hat G}(\lambda)$ contains all points with positive dimensional stabilizer. It is shown in \cite{Seppanen-GlobBranch} that, for embeddings of semisimple groups 
acting on complete flag varieties, the chambers are exactly the full-dimensional GIT-classes in $C^{\hat G}(X)$. The following theorem is perhaps 
known to experts, but we state it here since it is important in our setting, and present a proof in the text.\\

{\bf Theorem II:} The $\hat T$-ample cone on $X$ consists of the entire Weyl chamber and $\hat T$-chamber structure is defined 
by the hyperplanes $\mc H_{w\xi_j}$ orthogonal to $w\xi_j$ for $w\in W$ and $j=1,...,q$. The $\hat G$-chambers, whenever they exist, are convex cones, open in $\Lambda_\R$, spanned by certain unions of $\hat T$-chambers.\\

The Picard group of the quotient $Y_\lambda$ is naturally related to the Picard group of $X$, cf. \cite{KKV}. 
The relation becomes simpler when the unstable locus does not contain divisors. This motivates the definition of $\hat G$-movable GIT-classes as 
those whose unstable locus has codimension at least 2. The union of these classes forms a cone, called the $\hat G$-movable cone on $X$, denoted by
$$
{\rm Mov}^{\hat G}(X)=\{ \lambda\in \Lambda_\R^+: {\rm codim}_XX^{us}_{\hat G}(\lambda)\geq 2 \} \subset C^{\hat G}(X) \;.
$$

A $\hat G$-{\it movable chamber} is then a full-dimensional GIT-class $C$ with ${\rm codim}_X X^{us}(C) \geq 2$ and $X^{ss}(C)=X^{s}(C)$. In such a case, we obtain a 
geometric quotient $Y_\lambda$ whose Picard group embeds, via pullback followed by extension, as a sublattice of full rank in the Picard group of $X$, yielding an 
isomorphism over the reals. Such a quotient is shown in \cite{Seppanen-GlobBranch} to be a Mori dream space whose effective cone is identified with $C^{\hat G}(X)$. 
{\bf The question} arises: do $\hat G$-movable chambers exist, or under what conditions?

The requested chambers are not always present. An important class of counterexamples is supplied by spherical subgroups $\hat G\subset G$, where $\dim V_\lambda^{\hat G}\leq 1$: there are no $\hat G$-movable line bundles and the quotient is a point. There are also non-spherical cases, like $SL_2\subset SL_2^{\times 4}$, where the $\hat G$-movable cone is the diagonal ray $(\R_+)\rho$. In our previous work \cite{Seppa-Tsa-Principal} we have obtained detailed results for $\hat G$ a principal $SL_2$-subgroup of a semisimple group $G$; in this case $\hat G$-movable chambers exist if $\dim X\geq 5$, which for simple $G$ means not to be of type $A_2$ or $B_2$.\\

Using our formula for the unstable locus, we obtain a concrete description of a system of nested cones in $\Lambda_\R^+$ defined by codimension of the unstable locus, beginning with the $\hat G$-ample and the $\hat G$-movable cones.\\

\noindent{\bf Theorem III:} {\it The sets $C_k^{\hat G}(X)=\{\lambda\in\Lambda_\R^+:{\rm codim}_XX^{us}(\lambda)\geq k\}$, defined for $k\geq1$, form a finite sequence of nested rational polyhedral cones given by
$$
C_k^{\hat G}(X) = \{\lambda\in\Lambda_\R^+ :\;\; \lambda(w^{-1}\xi_j)\leq 0 \;\;, \;\forall j,w : \dim \hat GP_jx_w = \hat r_j+p_j(w)=\dim X-k+1\}.
$$
Moreover, the cone $C_{k+1}$ is contained in the interior of $C_{k}$, in the topology of the Weyl chamber.

The $\hat G$-ample and -movable cones are obtained for $k=1$ and $2$, respectively. Furthermore:

(i) The $\hat G$-ample cone is given by
$$
C^{\hat G}(X) = \{\lambda\in\Lambda_\R^+ :\;\; \lambda(w^{-1}\xi_j)\leq 0 \;\;,\forall j,w : \hat G\ol{P_jx_w} = X\}.
$$
For line bundles on its regular boundary, $\lambda\in\Lambda^{++}\cap \partial C^{\hat G}(X)$, one has unstable locus of codimension 1.

(ii) $\hat G$-movable chambers exist if and only if the cone $C_2^{\hat G}(X)$ has full dimension.

(iii) If $C_3^{\hat G}(X)\ne 0$, then $\hat G$-movable chambers exist.

(iv) The subgroup $\hat G \subset G$ is spherical if and only if $C_2^{\hat G}(X)=\{0\}$, i.e., $\lambda=0$ is the only dominant weight satisfying the 
inequalities $\lambda(w^{-1}\xi_j)\leq 0$, whenever $\dim\hat GP_jx_w =\dim X- 1$.}\\

A priori it is not clear that the codimension of the unstable loci could not make ``jumps'' and increase in steps bigger than one when passing from one GIT class in $C^{\hat G}(X)$ to another. The following ``no jumps'' result (cf. Lemma \ref{Lemma codim Jump 1}) is a key step in our proof of the above theorem, and presents an interest by itself.\\

\noindent{\bf No jump lemma:} Suppose that $C_1, C_2\subset C^{\hat G}(X)$ are GIT-classes in the $\hat G$-ample cone satisfying $\ol{C_1}\supset C_2$. Then
$$
| {\rm codim}_X X^{us}(C_1) - {\rm codim}_X X^{us}(C_2) | \leq 1 \;.
$$
The same inequality holds if $C_1, C_2$ are GIT-chambers sharing a facet.\\

The following corollary gives a numerical sufficient condition for presence of $\hat G$-movable chambers, expressed in terms of some structure constants of the embedding $\hat G\subset G$. It is obtained by considering the anticanonical bundle on $X$, i.e., $\lambda=2\rho$, the sum of the positive roots of $G$, which tends, heuristically, to have a small unstable locus.\\

\noindent{\bf Corollary:} For $j=1,...,q$, let $a_j$ and $b_j$ denote, respectively, the minimal and maximal positive value of a root of $G$ on $\xi_j$. If $\min\limits_j\{\frac{a_j}{a_j+b_j}r_j-\hat r_j\}\geq 2$, then $X$ admits $\hat G$-movable chambers.

In particular, $\hat G$-movable chambers exist for diagonal embeddings $\hat G\subset G=\hat G^{\times k}$ with sufficiently large $k$. If $\hat G$ is a product of classical groups, it suffices to take $k\geq 5$.\\\\

The second topic of this article concerns the quotients arising from $\hat G$-movable chambers, their Picard groups and Cox rings. Refining the aforementioned results \cite{Seppanen-GlobBranch} on the effective cone on the quotient, we find a natural identification between the GIT-equivalence relation in ${\rm Pic}(X)$ with the Mori equivalence relation in ${\rm Pic}(Y)$. The proofs appear in Theorems \ref{T: conesineff} and \ref{T: Mori-GITchambers}. The definition of Mori chambers can be found in Section \ref{Sect Mori chambers}, see also \cite{hk} for notions concerning Mori dream spaces.\\

\noindent{\bf Theorem IV:} {\it Suppose that there exists a $\hat G$-movable chamber $C\subset C^{\hat G}(X)$ and let $Y=Y_C$ be the corresponding GIT-quotient of $X$. Then $Y$ is a Mori dream space and there is a 
canonical isomorphism of $\R$-Picard groups giving rise to the following identifications:
\begin{gather*}
\begin{array}{rcl}
{\rm Pic}(X)_\R &\cong& {\rm Pic}(Y)_\R \\
C^{G}(X) &\cong& \ol{\rm Eff}(Y) \\
\textit{GIT-chambers} & \leftrightarrow & Mori\; chambers \\
{\rm Mov}^{\hat G}(X) & \cong & {\rm Mov}(Y) \\
\ol{C} & \cong & {\rm Nef}(Y) \\
{\rm Cox}(X)^{\hat G} \;\cong& \bigoplus\limits_{\lambda\in\Lambda^+} V_\lambda^{\hat G} &\cong \; \textit{Finite extension of}\; {\rm Cox}(Y)\;.
\end{array}
\end{gather*}
Moreover, all rational contractions of $Y$ to normal projective varieties are induced by VGIT from $X$.}\\

Let us note that the family of Mori dream spaces produced as GIT-quotients of flag varieties could be of independent interest. For the sake of representation theory, clearly a concrete model for $Y$ would be of great benefit -- as explained above, this variety would encode the full information on dimensions of $\hat G$-invariants in $G$-modules for the given subgroup $\hat G\subset G$. Although we are able to prove many nice properties, these spaces remain somewhat implicit,
as is often the case with quotients, due to the implicit nature of the fundamental existence results in invariant theory. The same is true to some extent for Mori dream spaces since several general constructions involve quotients, while many explicit alterations of varieties destroy the Mori dream property. It is therefore of interest to know whether our quotients appear among the known examples of Mori dream spaces. Perhaps the interaction of the Mori theory with the structure theory of semisimple groups 
could help to obtain more concrete information about this family of spaces, ideally build concrete models at least for special classes of subgroups like diagonals.

\section{The Hilbert-Mumford criterion and the Kirwan-Ness stratification}

Our approach to instability is based--as is often the case--on a fundamental result of Hilbert, reducing instability for linear actions of reductive groups to instability for their 
one-parameter subgroups, developed further by Mumford; cf. \cite{Mumfordetal-GIT}, for the general theory, and \cite{Ness-StratNullcone} for a shorter presentation suitable for our purposes.\\

\noindent{\bf Hilbert's theorem:} {\it Let $H\to GL(V)$ be a representation of a reductive complex algebraic group $H$. Then the ring of invariants $\C[V]^H$ is generated by a finite number of homogeneous elements. Let $J\subset \C[V]^H$ be the ideal vanishing at $0$ and let $V_H^{us}\subset V$ denote its zero locus, called the unstable locus. Then
$$
V_H^{us} = \{v\in V:\ol{Hv}\ni0\} =\{v\in V: \exists \gamma\in {\rm Hom}(\C^*, H): \lim\limits_{t\to0}\gamma(t)v=0\}\;.
$$
}\\

The homogeneity of the generators ensures that $\mP(V)_H^{us}$ is well defined. For a projective variety $Z\subset \mP(V)$ preserved by $H$, the zero locus of $J_{\vert Z}$ in $Z$ is obtained by intersection $Z_H^{us}=Z\cap\mP(V)_H^{us}$. The restriction of $\mc O_{\mP(V)}(1)$ to $Z$ is an equivariant ample line bundle $\mc L$. 

Mumford has devised a numerical criterion for instability for equivariant ample line bundles on projective varieties. We shall give a general statement, but in order to keep the notation in line we return to the case in hand, $H=\hat G, Z=X=G/B, \mc L=\mc L_\lambda, V=V_\lambda$, with some $\lambda\in\Lambda^{++}$ fixed for this section. This line bundle is ample, giving rise to a projective embedding obtained as the orbit of a highest weight vector:
$$
X\cong G[v^{\lambda}]\subset\mP(V_\lambda) \;.
$$
We identify the elements $\gamma\in{\rm Hom}(\C^*,G)$ by their infinitesimal generators in the Lie algebra $\xi=\dot\gamma(1)\in\mk g$ and call them one-parameter subgroups (OPS). We consider the $\xi$-unstable locus, taking the orientation into account:
$$
X^{us}_{\xi}(\lambda) = \{ [v]\in X: \lim\limits_{t\to-\infty}{\rm exp}(t\xi)v=0\}.
$$

Let us fix a pair of Cartan and Borel subgroups $\hat T\subset \hat B\subset \hat G$. The OPS of $\hat T$ form a lattice naturally identified with the dual to the weight lattice $\hat\Gamma=\hat\Lambda^\vee\subset\hat{\mk t}$. Recall that every OPS of $\hat G$ is conjugate to a unique element of $\hat\Gamma^+$, the set of dominant elements with respect to $\hat B$.

Mumford's numerical function for $\xi\in \hat\Gamma$,
\begin{gather}\label{For MumfordFunk}
M^\xi:X\to \Z \;,
\end{gather}
is defined as follows. For $x\in X$ let $x_0=\lim\limits_{t\to -\infty}{\rm exp}(t\xi)x \in X$. The limit point belongs to the fixed set of the OPS, $x_0\in X^\xi$. The connected components of $X^\xi$ are contained in the projectivizations of the eigenspaces of $\xi$. Define $M^\xi(x)$ to be the eigenvalue of $\xi$ at $x_0$. The point is $\xi$-unstable if the eigenvalue is positive.\\

\noindent{\bf Hilbert-Mumford criterion:} {\it Let $\mc L$ be a $\hat G$-equivariant ample line bundle on $X$, in the above notation $\mc L=\mc L_\lambda$, for some $\lambda\in\Lambda^{++}$. A point $x\in X$ is $\hat G$-unstable if and only if it is unstable for some one-parameter subgroup of $\hat G$. We have
$$
X^{us}_{\hat G}(\lambda) =\hat G X^{us}_{\hat T}(\lambda) = \hat G \left( \bigcup\limits_{\xi\in\hat\Gamma^+} X^{us}_{\xi}(\lambda)\right) \quad,\quad X^{us}_{\xi}(\lambda) = \{x\in X : M^\xi(x)>0 \}.
$$
}\\

To interpret the criterion in the specific situation of $X=G/B$, it is convenient to consider $T$ and $B$ containing $\hat T$ and $\hat B$,
respectively. Such extensions always exist, in general there are many, and this plays a role in our approach. The extensions of Borel subgroups correspond to the closed $\hat G$-orbits in $X$, which are all of the form $\hat G/\hat B$. The closed $\hat G$-orbits are all unstable for all $\lambda^{++}$. The extensions of Cartan subgroups allows to evaluate weights of $T$ on elements of $\hat T$, and compute with the criterion. We obtain $\iota:\hat\Gamma\subset \Gamma$ and $\iota^*\Lambda\to \hat\Lambda^+$. Notice that, for any nested pair of Borel subgroups we have $\iota:\Lambda^+\to\hat\Lambda^+$, but $\hat\Gamma^+$ is not necessarily contained in $\Gamma^+$. The hypothesis for our proof of the formula for the $\hat G$-unstable locus is that $\hat T$ contains regular elements of $G$ and hence determines a unique $T$. In this situation,
the eigenvalues of any $\xi\in\hat\Gamma$ in $V_\lambda$, and on $X^\xi$ in particular, are the values on $\xi$ of the $T$-weights.
To this end we introduce the following notation. Let $\mathcal{P}(\lambda)\subset \Lambda$ denote the set of $T$-weights of the irreducible $G$-module $V_\lambda$.
For any $x \in X$, let $\tilde{x} \in V_\lambda$ be a vector with $[\tilde{x}]=x \in X \subseteq \mathbb{P}(V_\lambda)$. 
We can decompose $\tilde{x}$ as a sum of weight vectors, 
\begin{align}
 \tilde{x}=\sum_{\mu \in \mathcal{P}(\lambda)} v_\mu. \label{E: weighdecompx} 
\end{align}
For $x \in X$, let $St(x) \subseteq \mathcal{P}(\lambda)$ denote the set of weights $\mu$ for which $v_\mu \neq 0$ in 
the decomposition \eqref{E: weighdecompx}. Then we have 
\begin{gather}\label{For Mxi = min St}
M^\xi(x)=\min\{\mu(\xi):\mu\in St(x)\} \;. 
\end{gather}

To compute the dimensions of the unstable loci, we shall need some general results from geometric invariant theory concerning stratifications of unstable loci. Specifically, the stratification theorems due to Kirwan in the symplectic setting and Ness for projective varieties, see \cite{Kirwan}, sections 12 and 13, and \cite{Ness-StratNullcone}, Theorem 9.5. More recently, Popov, \cite{Popov-Nullforms}, has refined the stratification results for (projective) representation spaces. It turns out that Popov's constructions can be applied successfully for complete flag varieties as well, as we show in Section \ref{Section Popov}.

The so-called Hesselink strata of $X^{us}$, as they are defined in \cite{Ness-StratNullcone} for any equivariantly embedded smooth $\hat G$-variety $X\subset\mP(V)$ with a $\hat G$-linearized $\mc O(1)$, have the form $\hat G X_{\xi,m}$, where $X_{\xi,m}$ is the so-called blade, determined by a one-parameter subgroup $\xi\in\hat \Gamma$ and a positive integer $m$ obtained as the value of a weight of a $\xi$-fixed point on $X$. Formally, any $\xi\in\hat\Gamma$ defines a eigenspace decomposition
$$
V=\bigoplus\limits_{m\in \Z} V^{\xi,m} \quad,\quad V^{\xi,m} = \{v\in V: \xi v = m v \} 
$$
The fixed point set in $X$ is then partitioned as
$$
X^\xi = \bigsqcup\limits_{m\in \Z} X^{\xi,m} \quad,\;{\rm where} \quad X^{\xi,m} = X\cap \mP(V^{\xi,m}) \;.
$$
The blade $X_{\xi,m}$ is defined as set set of points flowing into $X^{\xi,m}$ under $\xi_t={\rm exp}(t\xi)$ as $t\to -\infty$. In the projective situation the blades are obtained by intersection with the blades of the ambient projective space and are given by (cf. \cite{Popov-Nullforms})
$$
X_{\xi,m} = X\cap \mP(V)_{\xi,m} \;,\;\; \mP(V)_{\xi,m}=\mP(V^{\xi\geq m})\setminus\mP(V^{\xi>m}) \;,
$$
where $V^{\xi\geq m}$ denotes the sum of the eigenspaces with eigenvalue greater or equal to $m$ and similarly for $V^{\xi>m}$. Note that $\mP(V)_{\xi,m}$ is an orbit of the parabolic subgroup of $SL(V)$ defined by $\xi$ and limit set $\mP(V^{\xi,m})$ is an orbit of its Levi subgroup $SL(V)_\xi$, the centralizer of $\xi$. The blade $X_{\xi,m}$ is preserved by the parabolic subgroup $\hat P_\xi$ and the limit set $X^{\xi,m}$ is preserved by the centralizer of $\xi$, which is a Levi subgroup denoted by $\hat L_\xi$. There is a natural map
$$
\hat G\times_{\hat P_\xi} X_{\xi,m} \to \hat G X_{\xi,m} \;.
$$
By the Hilbert-Mumford criterion, $X_{\hat G}^{us}$ can be written as the union of $\hat G X_{\xi,m}$ over $\xi\in\hat\Gamma^+$ and $m>0$. The stratification theorems concern the existence of a finite number of blades, whose $\hat G$-saturation gives the entire unstable locus. Kirwan gives a characterization of these stratifying blades, referring to the connected components of $X^{\xi,m}$ and sets of semistable points in them.

\begin{defin}\label{Def StratElts}
A dominant one-parameter subgroup $\xi\in\hat\Gamma^+$ is called a stratifying element for $X_{\hat G}^{us}$, if it satisfies the following conditions:
\begin{enumerate}
 \item[(i)] $\xi$ is indivisible, i.e., $\frac1k \xi \notin \hat\Gamma$ for all $k>1$;
 \item[(ii)] there exists $m>0$ such that $(X^{\xi,m})_{\hat L_{\xi}/\xi}^{ss}(\mc O(1))\ne \emptyset$,
\end{enumerate}
The pairs $(\xi,m)$ for which (ii) holds are called stratifying pairs, and the corresponding blades $X_{\xi,m}$ - stratifying blades. We denote by $\mk S\subset\hat\Gamma^+$ the set of the stratifying elements.
\end{defin}

Note that the first condition is necessary just to remove the obvious redundancy arising from $X^{\xi,m}=X^{k\xi,km}$, while the second condition contains the essence of the notion. The next theorem states that the stratifying OPS are obtained as follows. For any set of weights $S\in \hat\Lambda$, consider the closest to zero point $\nu_S\in{\rm Conv}(S)\subset\hat\Lambda_\R$ and let $\xi_S\in\hat\Gamma$ denote the indivisible OPS generating the ray in $\hat{\mk{t}}$ corresponding, under the Killing from, to the ray of $\nu_S$. We set $\xi_S=0$ if $\nu_S=0$. Similarly, if $Z\subset X\subset \mP(V)$ is a subvariety preserved by $\hat T$, we denote by $St(Z^{\hat T})\subset\Lambda$ the set of weights of the $\hat T$-fixed set, and by $\xi_Z=\xi_{St(Z^{\hat T})}\in\hat\Gamma$ the OPS resulting from this set of weights. The following theorem has significant generalizations in both the symplectic and algebraic directions. We state it here in the projective setting, as we shall need it, and we adhere to our 
current notation, although the statement is general.

\begin{thm}\label{Theo Kirwan-Ness} {\rm (Kirwan, \cite{Kirwan}, Ness, \cite{Ness-StratNullcone})}

Let $X\subset \mP(V)$ be a smooth projective variety preserved by a reductive group $\hat G$ linearly represented on $V$. Let $St(X^{\hat T})$ be the set of weights of the fixed point set of a Cartan subgroup $\hat T\subset \hat G$. Let $\Xi=\{\xi_S: S\subset St(X^{\hat T})\}\cap (\hat\Gamma^+\setminus\{0\})$.
$$
X^{us}_{\hat G} = \bigsqcup\limits_{\xi\in\Xi, m\in \N} \hat G (X_{\xi,m})_{\hat L_\xi/\xi}^{ss} \;.
$$
For a nonempty stratum, the natural map $\hat G\times_{\hat P_\xi} (X_{\xi,m})_{\hat L_\xi/\xi}^{ss} \to \hat G (X_{\xi,m})_{\hat L_\xi/\xi}^{ss}$ is finite, and the dimension is
$$
\dim \hat G (X_{\xi,m})_{\hat L_\xi/\xi}^{ss} = \dim \hat G/\hat P_\xi + \dim X_{\xi,m} \;.
$$
\end{thm}

\begin{rem}
The union in the formula is disjoint, but the index set, as written, gives rise to empty summands. The obvious redundancy comes form $m$: clearly the values of the weights of $V$ on any fixed $\xi$ are bounded since $V$ is finite dimensional. A more subtle problem is presented by empty summands arising from blades with trivial semistable locus, and we shall see these manifest when $X$ is a flag variety.
\end{rem}

\section{Instability on flag varieties and Schubert varieties}

In this section we prove our formula for the unstable locus presented as Theorem I in list of main results. The key observation is that, when $X=G/B$, the connected components of the Hesselink blades are not only preserved by the parabolic subgroups of $\hat G$, but are in fact orbits of parabolic subgroups of $G$ through $T$-fixed points, in a way that $X_{\xi,m}=\cup P_\xi x_w$. union over $w\lambda(\xi)=m$. We begin by summarizing some facts about instability on flag varieties, which are either classical or can be found for instance in \cite{Beren-Sjam},\cite{Ressayre-2010-GITandEigen}. We include them since this will help us introduce relevant objects and notation, and because we would like to emphasize the simplicity of the argument.

\subsection{One-parameter subgroups}

For a semisimple element $\xi\in \mk g$ let $\mk l_\xi$ denote the 0-eigenspace of ${\rm ad}\xi$, i.e., the centralizer of $\xi$, let $\mk p_\xi$ denote sum of the 
nonnegative eigenspaces, and let $\mk r_\xi^{\pm}$ denote the sum of the positive/negative eigenspaces. Then $\mk p_\xi$ is a parabolic subalgebra of $\mk g$ with Levi decomposition $\mk p=\mk l_\xi+\mk r_\xi^+$. We denote the corresponding subgroups of $G$ by $P_\xi,L_\xi,R_\xi^{\pm}$ and also write $R_\xi=R_\xi^{+}$. Recall that an element $\xi\in\mk g$ is 
called regular if its centralizer is a Cartan subalgebra or, equivalently, $\mk p_\xi$ is a Borel subalgebra. We denote by $\mk g_{\rm reg}$ the set of regular semisimple elements, 
and for any subset $A\subset\mk g$ we denote $A_{\rm reg}=A\cap\mk g_{\rm reg}$. An element of our fixed Cartan subalgebra $\mk t$ is regular if it belongs to the interior of some 
Weyl chamber, i.e., no root in $\Delta=\Delta(\mk g,\mk t)$ vanishes on it.

\begin{lemma}\label{Lemma OPSandParabOrbit} Let $\xi\in \Gamma^+$ be a dominant OPS of $G$ with respect to a Borel subgroup $B$. Let $P_\xi=L_\xi R_\xi$ be the associated parabolic subgroup. Then the following hold:

\begin{enumerate}
\item[{\rm (i)}] The set of fixed points of $\xi$ in $X$ consists of the union of the closed $L_\xi$-orbits, which are exactly the $L_{\xi}$-orbits of the $T$-fixed points, or the $L_\xi$-orbits of the $B\cap L_\xi$-fixed points, parametrized by the left coset space $W_\xi\setminus W$, or by the set ${^\xi W}$ of shortest representatives:
$$
X^\xi = \bigcup\limits_{w\in W} L_\xi x_{w} = \bigsqcup\limits_{w\in {^\xi W}} L_\xi x_{w}
$$
\item[{\rm (ii)}] Every $P_\xi$-orbit in $X$ contains exactly one closed $L_\xi$-orbit. Every $P_\xi$-orbit contains a unique open $B$-orbit and a unique 
$B$-orbit of minimal dimension. These correspond to a pair of elements $w^1,w_1\in W_\xi w$ in every coset, having, respectively, maximal and minimal length with respect to $B$ 
related by $w^1=w_{01}w_1$, where $w_{01}$ is the longest element in $W_\xi$ with respect to $B_\xi= B\cap L_\xi$. The closure of every 
$P_\xi$-orbit is a Schubert variety $\ol{P_{\xi} x_w} = \ol{Bx_{w^{1}}}$. The dimension and codimension of an orbit are given by
$$
\dim P_\xi x = l(w^1) = n_\xi + l(w_1) \quad,\quad {\rm codim}_X P_\xi x = r_\xi - l(w_1) \;.
$$
where $n_\xi=\dim P_\xi/B$ and $r_\xi=\dim R_\xi$.

\item[{\rm (iii)}] For $\lambda\in\Lambda^+$, the Mumford function $M^\xi$ defined in (\ref{For MumfordFunk}) is constant on $P_\xi$-orbits and its values are given by the weights corresponding to the $T$-fixed points:
$$
M^\xi(x) = M^\xi(x_{w})=w\lambda(\xi) \quad{\rm for}\quad x\in P_\xi x_w \;.
$$
The $\xi$-unstable locus in $X$ with respect to $\lambda$ is given by
$$
X_{\xi}^{us}(\lambda) = \bigsqcup\limits_{w\in {^\xi W}^+(\lambda,\xi)} P_\xi x_{w} = \bigcup\limits_{w\in W^+(\lambda,\xi)_{B{\rm -max}}} \ol{B x_{w}} \;,
$$
where $W^+(\lambda,\xi) = \{w\in W : w\lambda(\xi)>0\}$. The first union is disjoint, while the second one gives exactly the irreducible components.
\end{enumerate}

\end{lemma}

\begin{proof}
For (i), since $L_\xi$ acts on $X^\xi$ and $X^T\subset X^\xi$ we have $L_\xi x_w\subset X^\xi$ for all $w$. On the other hand, we note that any fixed point $x\in X^\xi$ belongs to a unique Schubert cell $Bx_w$. Using the fact that $Bx_w=Nx_w$ and the global linearization of the $T$-action on $Nx_w$, one observes that $X^{\xi}\cap Nx_w = N_\xi x_w\subset L_\xi x_w$.

For (ii), note first that the $B$-orbits in a given $P_\xi$-orbit at the orbits through its $T$-fixed points. By irreducibility of orbit-closures every $P_\xi$-orbit contains a unique open $B$-orbit, say $\ol{Bx_{w^1}}=\ol{P_\xi x_{w^1}}$. On the other hand, computing the tangent spaces in terms of roots, one sees that for any $w\in W$, $Bx_w$ is open in $P_\xi x_w$ if and only if $N_\xi^-$ fixes $x_w$. There is a unique such point in every closed $L_\xi$-orbit, hence there is a unique closed $L_\xi$-orbit in every $P_\xi$-orbit, and the $T$-fixed points in $P_\xi x_w$ form a single $W_\xi$-orbit. By its definition $w^1$ has maximal length in $W_\xi w^1$, equal to the dimension of $P_\xi w^1$. One has $l(ww^1)=l(w^1)-l(w)$ for $w\in W_\xi$. The longest element $w_{01}\in W_\xi$ defines $w_1=w_{01}w^1$ of length $l(w_1)=l(w^1)-n_\xi$. This yield the dimension formulae.

For (iii), recall that the embedding $\phi_\lambda:X\hookrightarrow \mP(V_\lambda)$ defined by the line bundle $\mc L_\lambda$ sends $x_w$ to the extreme weight vector $[v_{w\lambda}]$. The Hilbert-Mumford criterion brings us to consider the following partition of the Weyl group (determined for any pair $(\lambda,\xi)\in\Lambda\times \Gamma$):
\begin{gather}\label{For Partition Wpm0lambda}
\begin{array}{l}
W = W^+(\lambda,\xi) \sqcup W^0(\lambda,\xi) \sqcup W^-(\lambda,\xi) \quad,\\
\quad W^+(\lambda,\xi) = \{w\in W : w\lambda(\xi)>0\} \;,\\
\quad W^0(\lambda,\xi) = \{w\in W : w\lambda(\xi)=0\} \;,\\
\quad W^-(\lambda,\xi) = \{w\in W : w\lambda(\xi)<0\} \;.
\end{array}
\end{gather}
If $w,w'\in W$ are related by the Bruhat order as $w'\leq w$, then $w'\lambda(\xi)\geq w\lambda(\xi)$ for all $\xi\in \mk t_+$. Indeed,
The Bruhat order is defined by $w'\leq w$ if $x_{w'}\in\ol{Bx_w}\subset X$, and $w'< w$ holds if $w'\ne w$. The linear span of the Schubert variety
in $V_\lambda$ is the Demazure $B$-module $V_{B,w\lambda}$ whose weights are exactly the weights of $V_\lambda$ contained in $w\lambda +Q_+$.
Thus $w'\lambda=w\lambda+q$ for some sum of positive roots $q$. If $h\in i\mk t_+$, then $q(h)\geq0$ and hence $w'\lambda(h)\geq w\lambda(h)$. 

Consequently, if $w$ belongs to either $W^+$ or $W^+\cup W^0$, then so do all elements smaller than $w$. Hence it suffices to take Bruhat-maximal elements as indices for the union.
\end{proof}

\subsection{Compatible Weyl chambers and cubicles}

Returning to our embedding $\hat G\subset G$, we have now two notions of regularity on $\hat{\mk g}$, its intrinsic one, and the one induced by the embedding in $\mk g$. We shall use the subscript reg only for the $G$-notion and use $\hat G$-reg for the intrinsic notion on $\hat{\mk g}$ or whenever more precision is necessary. So, for instance, $\hat\Gamma^+_{\rm reg}$ denotes the set of $\hat B$-dominant, $G$-regular OPS of $\hat T$. Clearly $G$-regular implies $\hat G$-regular, but in general the converse implication does not hold. 
It holds if and only if any Weyl chamber $\hat{\mk t}_+$ is contained in a unique Weyl chamber $\mk t_+$,

The calculation of the Mumford function in Lemma \ref{Lemma OPSandParabOrbit}, (iii), concerns a weight $\lambda$ and an OPS $\xi$ dominant with respect to the same Weyl chamber in $\mk t$. However, for the $\hat G$-unstable loci of the line bundles given by $\lambda$ in a given $\Lambda^+$, we need to handle OPS from $\hat\Gamma^+$, which is not necessarily contained in $\Gamma^+$. To this end we follow Berenstein and Sjamaar, \cite{Beren-Sjam}, who introduced the following notions associated in general to a pair $\hat{\mk g}\subset\mk g$ of reductive complex Lie algebras with a fixed pair $\hat{\mk t}\subset\mk t$ of nested Cartan subalgebras. Two Weyl chambers $\hat{\mk t}_+$ and $\mk t_+$ are called {\it compatible} 
if $\dim_\R \hat{\mk t}_+\cap \mk t_+=\dim_\C\hat{\mk t}$.
Let us fix from now on a compatible pair of chambers $\hat{\mk t}_+$ and $\mk t_+$.
The Weyl chambers of $\mk t$ are parametrized by Weyl group elements,
and the chambers compatible with $\hat{\mk t}_+$ determine the following set
$$
W_{\rm com} = \{w\in W: \dim_R(\hat{\mk t}_+\cap w\mk t_+) = \dim_C\hat{\mk t}\} \;,
$$
called {\it the compatible Weyl set}. For $\sigma\in W_{\rm com}$, the cone
$$
\hat{\mk t}_\sigma = \hat{\mk t}_+\cap \sigma\mk t_+
$$
is called a {\it cubicle} in $\hat{\mk t}$. We have
$$
\hat{\mk t}_+=\bigcup\limits_{\sigma\in W_{\rm com}} \hat{\mk t}_\sigma \;.
$$
Berenstein and Sjamaar observed that the Weyl group
of the centralizer $Z_G(\hat T)$ acts on $W_{\rm com}$, and defined {\it the relative Weyl set}
$W_{\rm rel}$ to be the set of shortest representatives of the respective coset space. 

\begin{prop}
For every nonzero $\xi\in\hat\Gamma^+$ there exists a unique element in $W_{\rm rel}$, to be denoted by $\sigma_\xi$, such that $\sigma_\xi^{-1}\xi\in\Gamma^+$.
\end{prop}

\begin{rem}
Under the hypothesis that $\hat G$ contains regular elements of $G$, the centralizer $Z_G(\hat T)$ equals the Cartan subgroup $T$ of $G$ and has trivial Weyl group $W_T(T)=\{1\}$. The compatible Weyl chambers are exactly those whose relative interiors intersect $\hat{\mk t}_+^\circ$. Hence we have
$$
W_{\rm com}=W_{\rm rel}=\{w\in W: \hat{\mk t}_+^\circ\cap w\mk t_+^\circ\ne\emptyset\} \;.
$$
Note that the Borel subgroups $B^\sigma$ for $\sigma\in W_{\rm com}$ contain the fixed $\hat B$, but they might not be all Borel subgroups of $G$ 
containing $\hat B$. More occur, for instance, for a root-subgroup $SL_2\subset SL_3$.
\end{rem}

Since $T=Z_G(\hat T)=Z_G(T)$ and $Z_G(\hat T)\subset N_G(\hat T)$ is a normal subgroup, we have $N_{\hat G}(\hat T)\subset N_G(\hat T)\subset N_G(T)$. This yields an inclusion
$$
{\rm j}:\hat W\subset W \;.
$$
The inclusion $\hat T\subset T$ is equivariant with respect to ${\rm j}$. There is a j-duality involution on $W$ given by $w\mapsto w^*={\rm j}(\hat w_0)w w_0$.
Lemma 2.4.3. in \cite{Beren-Sjam} states that $W_{\rm com}$ is stable under j-duality, and the cubicles are permuted by
$\hat{\mk t}_{\sigma^*}=-\hat w_0\hat{\mk t}_\sigma$ for $\sigma\in W_{\rm com}$.

Let $\sigma\in W_{\rm com}$ and $\lambda\in \Lambda^{++}$. Then $\iota^*(\sigma\lambda)\in\tilde\Lambda^+$ and hence $\sigma\in W^+(\lambda,\hat h)$
and ${\rm j}(\hat w_0)\sigma\in W^-(\lambda,\hat h)$ for any $\hat h\in\hat{\mk t}_+^\circ$. If $\hat h\in\hat{\mk t}_\sigma$, then $\sigma\lambda(\hat h)$
is the maximum of $W\lambda(\hat h)$.

\begin{defin}\label{Def xilengthofw}
For any $\xi\in\hat\Gamma^+\setminus\{0\}$ we fix a relative Weyl group element $\sigma_\xi\in W_{\rm rel}$ such that $\xi$ belongs to the cubicle $\hat{\mk t}_{\sigma_\xi}$. We denote $B_\xi=L_\xi\cap B^{\sigma_{\xi}}$, this is a Borel subgroup of $L_\xi$ such that $\hat B_\xi=B_\xi\cap \hat L_\xi$ is a Borel subgroup of $\hat L_\xi$. We denote by $l_\xi(w)=\dim B^{\sigma_\xi}x_{w\sigma_\xi}$ the $B^{\sigma_\xi}$-length of $w\in W$. We consider the left cosets of the stabilizer $W_\xi$ in $W$ and we denote by $^\xi W\subset W$ the set of $B^{\sigma_\xi}$-shortest representatives of the cosets. The set $^\xi W$ parametrizes the $B_\xi$-fixed points in $X$ and thus the closed $L_\xi$-orbits, $X^{B_\xi}=\{ x_{w\sigma_\xi} : w\in {^\xi W} \}$.

Given $\lambda\in\Lambda^+$, we denote, analogously to (\ref{For Partition Wpm0lambda}),
\begin{gather*}
\begin{array}{l}
W = W^+(\sigma_\xi\lambda,\xi) \sqcup W^0(\sigma_\xi\lambda,\xi) \sqcup W^-(\sigma_\xi\lambda,\xi) \quad,\\
W^+(\lambda,\xi) = \{w\in W : w\sigma_\xi\lambda(\xi)>0\} \;,\\
W^0(\lambda,\xi) = \{w\in W : w\sigma_\xi\lambda(\xi)=0\} \;,\\
W^-(\lambda,\xi) = \{w\in W : w\sigma_\xi\lambda(\xi)<0\} \;.\\
l_{\xi,\lambda}^+ = \max\{l_\xi(w): w\in {^\xi W}(\sigma_\xi\lambda,\xi)\} \;.
\end{array}
\end{gather*}
\end{defin}

\begin{rem}\label{Rem SingleCubicle}
The situation is somewhat simpler with regard of calculations and notation whenever it suffices to consider one cubicle, which means that any Weyl chamber of $\hat G$ is contained in some Weyl chamber of $G$, listed as property (a) below. It is useful to notice that this property is preserved for the Levi subgroups. More generally, consider following the properties (note that (d) $\Lw$ (c) $\Lw$ (a)+(b)):

(a) there exists a Weyl chamber of $G$ containing any given Weyl chamber of $\hat G$, i.e., $\hat\Gamma^+\subset\Gamma^+$, or equivalently $W_{\rm rel}=\{1\}$;

(b) containing regular elements;

(c) $\hat\Gamma^{++}\subset\Gamma^{++}$, i.e., $W_{\rm com}=\{1\}$ and there is a single cubicle;

(d) being a diagonal embedding of a semisimple group in a Cartesian power.

The definitions immediately imply the following:
\begin{enumerate}
\item[1)] If $\hat G\subset G$ has some of the properties (a),(b),(c),(d), then the same properties also hold for the natural embedding $\hat L'_\xi\subset L'_\xi$ of the semisimple parts of the centralizers of any semisimple element $\xi \in \hat{\mk g}$.
\item[2)] Suppose that $G_1\subset G_2\subset G$ is a chain of embeddings. The embedding $G_1\subset G$ has some of the properties (a),(b),(c),(d), if and only if the same properties hold for both embeddings $G_1\subset G_2$ and $G_2\subset G$. 
\end{enumerate}

Property (c) means, that the intrinsic notion of regularity for one-parameter subgroups of $\hat G$ coincides with that induced from the embedding in $G$, as mentioned above. Examples where this property is fulfilled are:

$\bullet$ diagonal embeddings $\hat G\subset G=\hat G^{\times k}$;

$\bullet$ $SL_2\subset SL_3$ given by any root;

$\bullet$ principal $SL_2$-subgroups $SL_2 \to G$ (characterized by having a single closed orbit in $G/B$);

$\bullet$ subgroup containing principal $SL_2$-subgroups $SL_2\to \hat G\subset G$, for instance $Sp_{2\ell}\subset SL_{2\ell}$ and $SO_{2\ell+1}\subset SL_{2\ell+1}$.\\
\end{rem}

\subsection{A formula for the unstable locus}

\begin{lemma}\label{Lemma XimaxisXicum}
Let $\mk P=\{P_\xi\subset G: \xi\in\hat\Gamma^+\setminus\{0\}\}$ denote the set of parabolic subgroups of $G$ defined by nonzero dominant OPS of $\hat G$. Let $\mk P_{\rm max}$ denote the set of maximal elements of $\mk P$. Let $\{\xi_1,...,\xi_q\}\subset\hat\Gamma^+$ denote the set of elements obtained as integral generators of rays of cubicles, and let $P_j=P_{\xi_j}$. Then $\mk P_{\rm max}=\{P_1,...,P_q\}$.
\end{lemma}

\begin{proof}
Recall that the parabolic subgroups of $G$ containing $T$ are determined by their (simple) roots. Let us fix a cubicle and denote $\hat\Gamma^+_{\sigma}=\hat\Gamma^+\cap\hat{\mk t}_{\sigma}$. For any two elements $\xi,\eta\in \hat\Gamma^+_{\sigma}$ belonging to a fixed cubicle, and hence to the same Weyl chamber of $G$, we have $P_{\xi+\eta}\subset P_{\xi}\cap P_{\eta}$. Hence the parabolic subgroups defined by the generating rays of the cone $\hat{\mk t}_\sigma$ are the maximal elements in the set of parabolic subgroups defined by $\xi\in\hat\Gamma^+_{\sigma}$. The elements of $\mk P$ defined by elements of $\hat\Gamma_\sigma^+$ are characterized by the property that of containing the Borel subgroup $B^\sigma$. Hence, if $\xi\notin\hat\Gamma_\sigma^+$, then $P_\xi$ does not contain $P_\eta$ with $\eta\in\hat\Gamma_\sigma^+$. Since every element of $\hat\Gamma^+$ is contained in some cubicle, we obtain that the maximal elements of $\mk P$ are exactly these defined by rays of cubicles.
\end{proof}

We denote $\Xi_{\rm max}=\{\xi_1,...,\xi_q\}$ and $\sigma_j=\sigma_{\xi_j}$.

\begin{thm}
For any $\lambda\in\Lambda^{++}$ the $\hat G$-unstable locus in $X=G/B$ can be written as
$$
X^{us}(\lambda) = \hat G \bigcup\limits_{j=1}^q X^{us}_{\xi_j}(\lambda) = \hat G \bigcup\limits_{j=1}^q \bigsqcup\limits_{w\in {^{\xi_j}W}: w\sigma_j\lambda(\xi_j)>0} P_j x_{w\sigma_j} 
$$
The codimensions of the $\xi_j$-unstable locus and the $\hat G$-unstable locus are bounded from below by
\begin{gather*}
\begin{array}{l}
{\rm codim}_X X_{\xi_j}^{us}(\lambda) \geq r_j-\hat r_j - l_{j,\lambda}^+ \\
{\rm codim}_X X^{us}(\lambda) \geq \min\limits_j \{r_j-\hat r_j - l_{j,\lambda}^+\} \;, 
\end{array}
\end{gather*}
where $l_{j,\lambda}^+ = l_{\xi_j,\lambda}^+$ is the number given in Definition \ref{Def xilengthofw}
\end{thm}

\begin{rem}
When a Weyl chamber of $\hat G$ is contained in a Weyl chamber of $G$ and we have the set $\Xi_{\rm max}$ consists simply of the fundamental coweights of $\hat G$, i.e., the generators of $\hat\Gamma^+$.
\end{rem}

\begin{proof}
From the Hilbert-Mumford criterion, we know that the $\hat G$-unstable locus is the $\hat G$-saturation of the union of the unstable loci for dominant OPS of $\hat G$. To reduce the instability with respect to an arbitrary $\xi\in\hat\Gamma^+$ to instability with respect to one of the $\xi_j$'s we shall use Lemma \ref{Lemma XimaxisXicum}. Let us take some $\xi\in\hat\Gamma^+\setminus\{0\}$ and $x\in X_{\xi}^{us}(\lambda)$. The $\xi$-unstable locus is described by Lemma \ref{Lemma OPSandParabOrbit}, and we conclude that $x\in P_\xi x_{w}$ for some $w\in W$ such that $w\lambda(\xi)>0$. The element $\xi$ belongs to some cubicle $\hat\Gamma_\sigma^+$ and can be expressed as a linear combination of the generators of this cubicle, say $\xi_1,...,\xi_p$, with nonnegative coefficients. We can deduce that $w\lambda(\xi_j)>0$ for some $j\in\{1,...,p\}$ for which the coefficient of $\xi$ is nonzero. Hence $P_\xi\subset P_j$ and we have $x\in P_\xi x_w\subset P_j x_w \subset X_{\xi_j}^{us}(\lambda)$. This proves the first formula for $X^{us}(\lambda)$. The second formula is deduced directly from Lemma \ref{Lemma OPSandParabOrbit} and Definition \ref{Def xilengthofw}.

The bound on the codimension follows from the standard fact that, for any $\xi\in\hat\Gamma^+$, the parabolic subgroup $\hat P_{\xi}\subset\hat G$ satisfies $\hat P_{\xi}=P_\xi\cap\hat G$, hence it acts every $P_\xi$-orbit $P_\xi x$, and we have a surjective map
$$
\hat G \times_{\hat P_\xi} P_\xi x \to \hat GP_\xi x \;.
$$
The dimension of the fibre bundle is $\hat r_j+\dim P_\xi x$ and hence this number bounds the dimension of $\hat GP_\xi x$ from above. We take a maximal dimensional $P_\xi x_{w\sigma_\xi}$ inside $X_{\xi}^{us}(\lambda)$, with $w\in {^\xi W}$. We may apply the codimension formula of Lemma \ref{Lemma OPSandParabOrbit}, part (ii), to the length function $l_\xi$ referring to the cubicle of $\xi$ (see Definition \ref{Def xilengthofw}). We obtain ${\rm codim}_X P_\xi x_{w\sigma_\xi} = r_\xi - l_\xi(w)=r_\xi-l_{\xi,\lambda}^+$. This completes the proof.
\end{proof}

We can deduce from the above proof the following expression for the $\hat T$-unstable locus. Note that unstable loci for tori are just a particular case of the above theorem, or, as well of the corollary. What we state below concerns Cartan subgroups of reductive groups and uses the Weyl group action on $X_{\hat T}^{us}(\lambda)$.

\begin{cor}\label{Cor hatT-unstablelocus}
The $\hat T$-unstable locus for a given $\lambda^{++}$ is given by
$$
X_{\hat T}^{us}(\lambda) = \bigcup\limits_{\hat w\in\hat W} \bigcup\limits_{j=1}^q X^{us}_{\hat w\xi_j}(\lambda) \;.
$$
Its codimension is given by
$$
{\rm codim}_X X_{\hat T}^{us}(\lambda) = \min\{r_j-l_{j,\lambda}^+: j=1,...,q\} \;.
$$
\end{cor}

We can also deduce the following expression for the $\hat G$-unstable locus in $X=G/B$ as a saturation of a union of Schubert cells.

\begin{cor}
Let $\hat G\subset G$ be a reductive subgroup of a semisimple group $G$ and let $\lambda\in\Lambda^{++}$ be a strictly dominant weight, defining an ample line bundle on $X=G/B$. Then the $\hat G$-unstable locus $X^{us}(\lambda)$ is the $\hat G$-saturation of the unstable Schubert varieties for the Borel subgroups $B^{\sigma}$ with $\sigma\in W_{\rm rel}$:
$$
X^{us}(\lambda) = \hat G \bigcup\limits_{(w,\sigma)\in W_{\rm rel}\times W: w\sigma\lambda(\hat\Gamma^+)\cap\Z_+\ne\emptyset} B^{\sigma} x_{w\sigma}
$$
In particular, if $W_{\rm rel}=\{1\}$ it suffices to consider a single Borel subgroup and one has
\begin{gather}\label{For SingCubiKirwanStrat}
X^{us}_{\hat G}(\lambda) = \bigcup_{w\in W\setminus W^{0-}(\lambda)} \hat G Bx_{w} \quad, 
\end{gather}
where
\begin{align*}
& W^\pm(\lambda)=\{w\in W: w\lambda(\hat{\mk t}_+^\circ)\subset \R_\pm \} \\
& W^{0\pm}(\lambda)=\{w\in W: w\lambda(\hat{\mk t}_{\pm}^\circ)\subset \R_{\geq 0} \} \;.
\end{align*}
\end{cor}

\begin{proof}
The Hilbert-Mumford criterion reduces the instability for $\hat G$ to instability of its dominant OPS, i.e., elements of $\hat\Gamma^+$. Every such OPS is contained in some cubicle with a corresponding $\sigma\in W_{\rm rel}$. By Lemma \ref{Lemma OPSandParabOrbit} the unstable locus of $\xi\in\hat\Gamma^+\cap\hat{\mk t}_\sigma$ is a union of $B^\sigma$-orbits.
\end{proof}

\section{The Kirwan-Ness stratification of the unstable locus}

To compute the dimensions of the unstable loci, we shall use the Kirwan-Ness stratification (Theorem \ref{Theo Kirwan-Ness}) applied to the case $X=G/B$ embedded in $\mP(V_\lambda)$ by the ample line bundle corresponding to some $\lambda\in\Lambda^{++}$. The one-parameter subgroups of any subgroup $\hat G\subset G$ are also one-parameter subgroups of $G$, and Lemma \ref{Lemma OPSandParabOrbit} provides a description of the resulting blades as orbits of parabolic subgroups of $G$. The lemma concerns $\xi\in\Gamma$ dominant with respect to the same Weyl chamber as $\lambda$. The set $\hat\Gamma^+$ of dominant OPS of $\hat G$ is partitioned by the cubicles, and for $\xi\in\hat\Gamma^+\cap\hat{\mk t}_\sigma$ we obtain
\begin{gather}\label{For BladesInFlags}
X^{\xi,m}(\lambda) = \bigsqcup\limits_{w\in {^\xi W}:w\sigma\lambda(\xi)=m} L_\xi x_{w\sigma}  \;,\quad X_{\xi,m}(\lambda) = \bigsqcup\limits_{w\in {^\xi W}:w\sigma\lambda(\xi)=m} P_\xi x_{w\sigma} \;.
\end{gather}

It remains to determine the stratifying pairs $\xi,m$. According to the stratification theorem, we consider the OPS determined by averaging weights of $\hat T$-fixed points in $X^{\xi,m}(\lambda)$. In our case the set of $\hat T$-weights is the projection of the set of $T$-weights, which in turn is a union of $W_\xi$-orbits:
\begin{gather}\label{For BladesInFlags-Weights}
St_T(X^{\xi,m}(\lambda)) = \bigsqcup\limits_{w\in {^\xi W}: w\sigma\lambda(\xi)=m} W_\xi w\sigma\lambda  \;,\quad St_{\hat T}(X^{\xi,m}(\lambda))=\iota^*(St_T(X^{\xi,m}(\lambda))) \;.
\end{gather}
The stratification theorem and the above observations bring us to the following definitions.

\begin{defin}\label{Def OPSforLeviOrbits XiWlambda}

(1) Let $\mk L=\{L_\xi : \xi\in\hat\Gamma^+\setminus\{0\}\}$ be the set of centralizers in $G$ of nonzero dominant one-parameter subgroups of $\hat T$.

(2) For any triple $(L,w,\lambda)\in {\mk L}\times W \times \Lambda$, let $\nu_{L,w,\lambda}\in\hat \Lambda_\R$ denote the closest to $0$ point in ${\rm Conv}(\iota^*(W_Lw\lambda))$. If $\nu_{L,w,\lambda}\ne 0$, let $\xi_{L,w,\lambda}\in\hat\Gamma$ be the indivisible integral generator of the ray in $\hat{\mk{t}}$ corresponding, under the Killing form, to the ray of $\nu_{L,w,\lambda}$ in $\hat\Lambda_\R$. In case $\nu_{L,w,\lambda}\ne 0$, we put $\xi_{L,w,\lambda}=0$.

(3) For $\lambda\in\Lambda^+$, denote
\begin{align*}
& \Xi_\lambda = \{\xi_{L,w,\lambda}\in\hat\Gamma : (L,w)\in {\mk L}\times W\} \;,\\
& \Xi_\lambda^+=\Xi_\lambda\cap\hat\Gamma^+ \;, \\
& \Xi'_\lambda=\{\xi\in\Xi_\lambda^+\setminus \{0\}: L_\xi\in\mk L'\} \;, \\
& \Xi\mk W_\lambda^+ = \{ (\xi,w)\in \Xi_\lambda^+\times W : w\in{^\xi W}^+(\sigma_\xi\lambda,\xi) \} \;.
\end{align*}
\end{defin}

\begin{rem}\label{Rem XimaxisXicub}
Note that $\Xi_{\rm max}\subset\Xi_\lambda^+$ for every $\lambda\in\Lambda^{++}$. The Levi subgroups $L_j=L_{\xi_j}$, $j=1,...,q$, are exactly the maximal elements of $\mk L$, i.e., 
\begin{gather}\label{For LeviMax}
\mk L_{\max}=\{L_{\xi}: \xi\in\Xi_{\rm max}\} \;.
\end{gather}
We denote $L_j=L_{\xi_j}$ the Levi component of $P_j$ for $j=1,,,.q$. For every $L_j$, the intersection $Z(\mk l_j)\cap \hat{\mk g}$ of the center of $\mk l_j$ with $\hat{\mk g}$ is one-dimensional, generated by $\xi_j$. Thus, for every $(w,\lambda)\in W\times \Lambda$ the element $\xi_{L_j,w,\lambda}$ is proportional to $\xi_j$ and in fact
$$
\xi_{L_j,w,\lambda} \in \{\xi_j,0,-\xi_j\} \;.
$$
\end{rem}

\begin{lemma}\label{Lemma SW-StratifyingPairs}
Let $\lambda\in\Lambda^{++}$. The set of dominant stratifying OPS for the $\hat G$-unstable locus $X^{us}(\lambda)$ is given by 
$$
\mk S_\lambda=\{\xi\in\Xi'_{\lambda} : \exists w\in W : w\sigma_\xi\lambda(\xi)>0, w\sigma_\xi\lambda\in C^{\hat L'_\xi}(L'_\xi x_{w\sigma_\xi}) \} \;.
$$
Denote
$$
{\mk S\mk W}_\lambda = \{(\xi,w)\in\mk S_\lambda\times W : w\in {^\xi W}^+(\sigma_\xi \lambda,\xi), ( L_\xi [v_{w\sigma_\xi\lambda}])_{\hat L'_\xi}^{ss} \ne \emptyset \}.
$$
Then there is a natural surjective map $\mk S\mk W_\lambda\to \wt{\mk S}_\lambda$ onto the set of stratifying pairs, given by $(\xi,w)\to(\xi,w\sigma_\xi\lambda(\xi))$.
\end{lemma}

\begin{proof}
The lemma follows from the Kirwan-Ness stratification theorem and Lemma \ref{Lemma OPSandParabOrbit}. Indeed, we were led to the definition of the set $\Xi'_\lambda$ by applying to our $X=G/B$ the constructions of the blades need in Definition \ref{Def StratElts} of stratifying elements. The property which remains to be checked is the presence of Levi-semistable points in the blades. Our blades are parametrized by $\Xi_\lambda^{+}\setminus \{0\}$, and we have to check the condition
$$
(L_\xi [v_{w\sigma_\xi\lambda}])_{\hat L_\xi/\xi}^{ss} \ne \emptyset
$$
with $\xi\in\Xi_\lambda$ and $w\in {^\xi W}^+(\lambda,\xi)$. The above condition may be expressed as: the restriction of the line bundle $\mc L_\lambda$ from $X$ to $L_\xi x_{w\sigma_\xi}$ is $\hat L_\xi/\xi$-ample. The fact that $\xi$ is of the form $\xi_{L,w,\sigma_\xi\lambda}$ ensures that the $\hat T/\xi$-semistable locus is nonempty, so the line bundle is $\hat L_\xi/\xi$-ample if and only if it is $\hat L'_\xi$-ample. Since $w$ is the shortest representative in its left $W_\xi$-coset, the weight $w\sigma_\xi\lambda$ (or rather its appropriate restriction) is dominant with respect to the Borel subgroup $B^{\sigma_\xi}\cap L'_\xi$ of $L'_\xi$. So the requested semistable locus is nonempty if and only if $w\sigma_\xi\lambda\in C^{\hat L'_\xi}(L'_\xi/(B^{\sigma_\xi}\cap L'_\xi))$, which is just the condition imposed in the definition $\mk S_\lambda$. For the second statement of the lemma it remains to notice that the requirement for $w$ to be the shortest representative in its $W_\xi$-coset ensures a 
bijective correspondence between $\mk S\mk W_\lambda$ and the set of connected components of Kirwan-Ness strata.
\end{proof}

\begin{thm}\label{Theo KirwanStratFlag}
Let $\lambda\in\Lambda^{++}$. The Kirwan-Ness stratification of $\hat G$-unstable locus in $X=G/B$ with respect to the line bundle $\mc L_\lambda$ is given by
$$
X^{us}(\lambda) = \bigsqcup\limits_{(\xi,w)\in \mk S\mk W_\lambda} \hat G (P_\xi x_{w\sigma_\xi})_{\hat L_\xi/\xi}^{ss}(w\sigma_\xi\lambda) \;.
$$
The dimension and codimension of the stratum for $(\xi,w)\in\mk S\mk W_\lambda$ are given by
\begin{align*}
& \dim \hat G (P_\xi x_{w\sigma_\xi})_{\hat L_\xi/\xi}^{ss}(w\sigma_\xi\lambda) = \dim\hat G/\hat P_\xi+\dim P_\xi x_{w\sigma_\xi}= \hat r_\xi + n_\xi    + l_\xi(w) \;, \\
& {\rm codim}_X \hat G P_\xi x_{w\sigma_\xi} = r_\xi - \hat r_\xi - l_\xi(w) \;.
\end{align*}
The dimension and codimension of the unstable locus are given by
\begin{align*}
&\dim X^{us}(\lambda) = \max\{ \hat r_\xi+n_\xi+l^{\rm str}_{\xi,\lambda} : \xi\in \mk S_\lambda \} \;, \\
&{\rm codim}_X X^{us}(\lambda) = \min\{r_\xi-\hat r_\xi-l^{\rm str}_{\xi,\lambda} : \xi\in \mk S_\lambda\}  \;,
\end{align*}
where $l^{\rm str}_{\xi,\lambda} = \max\{l_\xi(w):(\xi,w)\in \mk S\mk W_\lambda\}$.
\end{thm}

\begin{proof}
The formula for the unstable locus follows from Lemma \ref{Lemma SW-StratifyingPairs}. The dimension formulae follow from the Kirwan-Ness dimension formula for the strata and Lemma \ref{Lemma OPSandParabOrbit}. 
\end{proof}

\begin{rem}
In Section \ref{Section Popov}, based on ideas of Popov, \cite{Popov-Nullforms}, we present an algorithm, using rooted trees, giving a ``yes'' or ``no'' answer to the question whether a given $\lambda$ belongs to $C^{\hat T}(X)$. This algorithm can be applied to the Levi subgroups $\hat L'_\xi\subset L_\xi'$ and any given $w\sigma_\xi\lambda$, in order to determine completely the Kirwan-Ness stratification in any given case.
\end{rem}




\section{The $\hat G$-ample cone of $G/B$ and GIT-classes}

The codimension formula for the $\hat G$-unstable locus of $\mc L_\lambda$ in $X$ provides a description of the $\hat G$-ample cone $C^{\hat G}(X)$ by linear inequalities. In fact we have similar descriptions obtained simultaneously for all $C_{k}^{\hat G}(X)$, showing that these form a sequence of rational polyhedral cones in $\Lambda^+_\R$, each contained in the relative interior of the previous one, in the relative topology of the Weyl chamber. We need some preliminary results on GIT-classes with respect to the torus $\hat T$, which form a subdivision of the GIT-classes for $\hat G$.

\subsection{GIT-chambers for $\hat T$ and $\hat G$}

Recall that a GIT-class on $X$ with respect to a given subgroup of $G$ is a chamber if the unstable locus is a proper subvariety and contains all points with positive dimensional stabilizer. In the flag variety, the connected components of the fixed point set of a one-parameter subgroup are the closed orbits of its centralizer. We have the following.

\begin{lemma}\label{Lemma T-chambers}
Let $\mk L_{\max}=\{L_1,...,L_q\}$ be the set of maximal elements in $\mk L$ with respect to inclusion (see Remark \ref{Rem XimaxisXicub}). For any $\lambda\in \Lambda^{++}$ the following are equivalent:
\begin{enumerate}
\item[{\rm (i)}] the $\hat T$-GIT-class of a weight $\lambda$ is a $\hat T$-chamber;
\item[{\rm (ii)}] $0\notin\Xi_{\lambda}$, i.e. the closed orbits $L_\xi x_w$ in $G/B$ of centralizers in $G$ of nontrivial OPS of $\hat T$ are $\hat T$-unstable; 
\item[{\rm (iii)}] $\xi_{L_j,w,\lambda}\ne 0$ for all $w\in W$ and $j=1,...,q$.
\item[{\rm (iv)}] $\lambda(w\xi_j)\ne0$ for all $w\in W$ and $j=1,...,q$.
\end{enumerate}
\end{lemma}

\begin{proof}
By Lemma \ref{Lemma OPSandParabOrbit} the fixed point set of of any OPS is the union of the closed orbits of its centralizer. So having a $\hat T$-chamber is equivalent to having all closed Levi orbits unstable, which is in turn equivalent to (ii).

To see that (iii) implies (ii) recall that $\xi_{L,w,\lambda}$ is the indivisible OPS corresponding to the weight $\nu_{L,w,\lambda}$, which is the closest to $0$ point in the convex hull of $\iota^*(W_Lw\lambda)$, and hence if 
$\xi_{L,w,\lambda}=0$ for some $L\in \mk L$ and $w\in W$, then $\xi_{L_j,w,\lambda}=0$ for any $L_j\supset L$.

The equivalence of (iii) and (iv) follows form Remark \ref{Rem XimaxisXicub}.
\end{proof}

\begin{thm}\label{Theo hatT-GIT-classes}
The decomposition of the ample cone $\Lambda^+_\R$ on $X$ into GIT-classes with respect to $\hat T$ is defined by the following system of hyperplanes, parametrized by pairs $\xi\in\Xi_{\rm max}$, $w\in {^\xi W}$,
$$
\mc H_{\sigma_\xi^{-1} w^{-1}\xi} = \{ \lambda\in\Lambda_\R : \lambda(\sigma_\xi^{-1} w^{-1}\xi) = 0 \} \;,
$$
\end{thm}

\begin{proof}
We described the $\hat T$-unstable locus for $\lambda\in\Lambda^{++}$ in Corollary \ref{Cor hatT-unstablelocus}. By Lemma \ref{Lemma T-chambers} (specifically (i)$\tst$(iii)), the walls bounding the $\hat T$-chambers are indeed defined by hyperplanes of the above form.
\end{proof}

\begin{thm}\label{Theo hatG-chambersandfacets}
Let $C$ be a $\hat T$-chamber consider a regular facet $F\subset \ol{C}$, which is necessarily of the form $\ol{F}=\ol{C}\cap\mc H_{\sigma_\xi^{-1} w^{-1}\xi}$ according to Theorem \ref{Theo hatT-GIT-classes}. Then $C$ and $F$ define distinct GIT-classes with respect to $\hat G$ if and only if $w\sigma_\xi F\subset C^{\hat L'_\xi}(L_\xi x_w)$.
\end{thm}

\begin{proof}
If the condition $w\sigma_\xi F\subset C^{\hat L'_\xi}(L_\xi x_w)$ implies, by Theorem \ref{Theo KirwanStratFlag}, that $\hat G P_\xi x_w$ (or rather an open subset of it) is a Kirwan-Ness stratum for $\lambda\in F$. 
\end{proof}

\begin{example}
The following example shows that the $\hat G$-classes do not coincide with the $\hat T$-classes. Consider the diagonal embedding $\hat G = SL_3 \hookrightarrow SL_3^{\times 3} = G$. Let $\lambda=(\hat\lambda,\hat\lambda+\hat\lambda^*,\lambda^*)$, with any $\hat\lambda=\binom{a_1}{a_2}$ satisfying $a_1>a_2>0$. Then we have $\lambda\in C^{\hat G}(X)\cap\mc H_{w^{-1}\hat\epsilon_1}$, where $w=(s_1s_2,s_1,1)$. The semisimple centralizer subgroups are then given by $\hat L'_1=SL_2\hookrightarrow SL_2^{\times 3} = L'_1$, and the variety $Z=L_1x_w$ is a triple product $(\mP^1)^{\times3}$. The element $w$ is of minimal length in its coset by $W_1=\{1,s_2\}^{\times 3}$. Further, we calculate that
$$
\lambda_1=w\lambda_{\vert \mk t\cap\mk L'_1}=\frac13(a_1-a_2,3a_1+3a_2,2a_1+a_2) \;.
$$
The middle coordinate of this weight, $a_1+a_2$, exceeds the sum of the other two coordinates, which is $a_1$. Hence, from our knowledge of the $SL_2$-ample cone for diagonal embeddings, we deduce that $\lambda_1\notin C^{\hat L'_1}(Z)$. (Formally, we describe the $\hat G$-maple cone on the next section, so this example could wait, but the case of $SL_2\subset SL_2^{\times 3}$ this follows simply from the Clebsch-Gordon rule.) Hence $\hat G P_{\hat\epsilon_1} x_w$ is not a Kirwan-Ness stratum for $\lambda$ and, by continuity, for weights in a neighbourhood of $\lambda$ in $\Lambda_\R$.
\end{example}

\subsection{The $\hat G$-ample and -movable cones}

We have seen that the Kirwan-Ness strata of the unstable locus are of the form $\hat GP_\xi x_{w\sigma_\xi}$ for stratifying pairs $(\xi,w)\in \hat\Gamma^+\times W$. The dimension of such stratum is given by $\dim \hat G P_\xi x_{w\sigma_\xi}=\hat r_\xi+\dim P_\xi x_w$. It is convenient to consider the pairs $(\xi,w)$ for which the dimension condition is satisfied; note that this property concerns just the action of $\hat G$ on $G/B$ and does not refer to any $\lambda$.

\begin{defin}
We call a pair $(\xi,w)\in \hat\Gamma^+\times W$ of a dominant OPS of $\hat G$ and a Weyl group element of $G$ a fit pair, if $\dim \hat R_\xi^-P_\xi x_{w\sigma_\xi}=\dim P_\xi x_{w\sigma_\xi} + \dim \hat R_\xi^-$. We denote the set of fit pairs, with the additional requirement that $w$ is the $B^{\sigma_\xi}$-shortest element in its left $W_\xi$-coset, by 
$$
\Xi\mk W_{\rm fit} = \{(\xi,w)\in \hat\Gamma^{+}\times W: w\in {^\xi W}, {\rm codim}_X \hat G P_\xi x_{w\sigma_\xi} = r_\xi-\hat r_\xi - l_\xi(w)\} \;.
$$
For fixed $\xi\in \hat\Gamma^+$, we denote by ${^\xi W}_{\rm fit}$ the set of elements in ${^\xi W}$ forming a fit pair $(\xi,w)$; for $l\in \N$, we denote ${^{\xi}W}_{\rm fit}(l)$ the subset with $l_\xi(w)=l$.
\end{defin}

\begin{rem} Let $\Delta=\Delta^+\sqcup\Delta^-$ be a root system split into positive and negative parts. It is well known that Weyl group elements are uniquely determined by the sets of positive roots they send to negatives. For $w\in W$, the set $\Phi_w=\Delta^+\cap w^{-1}\Delta^-$ is called the inversion set and the set $\Psi_w=\Delta^-\cap w\Delta^+$ the inverted set. We have $l(w)=\#\Psi_w$ and $\Psi_w=-\Phi_{w^{-1}}$. For a given $\xi\in\Gamma^+\setminus\{0\}$, the decomposition $\Delta=\Delta(\mk l_\xi)\sqcup\Delta(\mk r_\xi^+)\sqcup\Delta(\mk r_\xi^-)$ is invariant under the action of $W_\xi$. An element $\tau\in W_\xi$ is determined by its relative inverted set $\Delta^-(\mk l_\xi)\cap w\Delta^+(\mk l_\xi)$. Also, we have $\Psi_{\tau w}\cap\Delta(\mk r_\xi^-)=\tau(\Psi_w\cap\Delta(\mk r_\xi))$; in particular, these to sets have the same cardinality. Hence the shortest element in a coset $W_\xi w$ is characterized by the property $\Psi_{w}\subset \Delta(\mk r_\xi^-)$, while the longest is 
characterized by $\Psi_{w}\supset \Delta^-(\mk l)$.
\end{rem}

\begin{lemma}\label{Lemma InvertedSets}
Let $\Delta=\Delta^+\sqcup\Delta^-$ be a root system split into positive and negative parts and let $w\in W$. Then there exists an order on the inverted set $\Psi_w=\{\beta_1,...,\beta_l\}$ such that, upon setting $w_j=s_{\beta_j} \dots s_{\beta_l}$, for $j=1,...,l$ and $w_{l+1}=1$, one gets
$$
w=w_1=s_{\beta_1} \dots s_{\beta_l} \;,\; \Psi_{w_{j+1}}=\Psi_{w}\setminus\{\beta_{1},...,\beta_j\} \;, l(w_{j+1})=l-j. \;
$$
Moreover, the root $\beta_j$ is simple for $w_j\Delta^+$.
\end{lemma}

\begin{proof}
We shall proceed by induction on the length $l=l(w)$. Let $\Pi$ be the set of simple roots in $\Delta^+$, so that $w\Pi$ is the set of simple roots in $w\Delta^+$. Let $\beta=\beta_1\in w\Pi\cap\Delta^-$; such an element exists, as long as $l>0$. Consider $w_2=s_{\beta}w$. Note that $-\beta\in w\Delta^-\cap\Delta^+$, so that, if $U_{-\beta}\subset B$ denotes the one-parameter unipotent subgroup of the root $-\beta$, then $\ol{U_{-\beta}x_w}=U_{-\beta}x_w\sqcup\{x_{w_2}\}$ and so $\ol{Bx_w}\supset Bx_{w_2}$. Since $\beta$ is simple for $w\Delta^+$, we have $s_\beta w\Delta^+\cap w\Delta^-=\{-\beta\}$. Hence $\{\beta\} = s_\beta w\Delta^- \cap w\Delta^+= (\Delta^-\setminus\Psi_{s_\beta w})\cap \Psi_w$. On the other hand,
\begin{align*}
\Psi_{s_\beta w} & = s_\beta w \Delta^+\cap\Delta^- \\
 & = (s_\beta w\Delta^+\cap\Delta^- \cap w\Delta^-) \cup (s_\beta w\Delta^+\cap\Delta^-\cap w\Delta^+) \\
 & = \emptyset \cup \Psi_{s_\beta w}\cap\Psi_{w}  \subset \Psi_{w} \;.
\end{align*}
Thus $\Psi_w = \{\beta\}\cup\Psi_{w_2}$. By induction on $l$ based on the trivial case $l=0$, i.e., $w=1$, we obtain the statement of the lemma.
\end{proof}

\begin{lemma}\label{Lemma SeqFit}
Let $(\xi,w)\in \Xi\mk W_{\rm fit}$ be a fit pair and $l=l_\xi(w)$. Then there exists a sequence $w=w_1...,w_{l+1}=1$ in ${^\xi W}_{\rm fit}$ with $l_\xi(w_{j+1})=l-j$ and
$$
\ol{P_\xi x_{w_j\sigma_\xi}}\supset P_\xi x_{w_{j+1}\sigma_\xi} \quad,\quad {\rm codim}_X \hat G P_\xi x_{w_j\sigma_\xi} = r_\xi - \hat r_\xi - l+j+1 \;. 
$$
If $\lambda\in\Lambda^{++}$ and $w\sigma_\xi\lambda(\xi)\geq 0$, then $w_j\sigma_\xi\lambda(\xi)>0$ for $j\geq 2$.
\end{lemma}

\begin{proof}
The proof is based on Lemma \ref{Lemma InvertedSets}. We apply it to $w$ with respect to the system of positive roots $\sigma_\xi\Delta^+$ associated to $\xi$. For the first part of the lemma, we may assume, without loss of generality, that $\sigma_\xi=1$, so that $\xi\in\Gamma^+\cap\hat\Gamma^+$. The case $l=0$ being trivial, we assume $l\geq 1$. Note that the condition $w\in {^\xi W}$ is equivalent to $\Psi_{w}\subset \Delta(\mk r_\xi^-)$. Hence, if $w=w_1,...,w_{l+1}=1$ is a sequence as obtained from Lemma \ref{Lemma InvertedSets}, then $\emptyset\subset \Psi_{w_{l}}\subset ... \subset\Psi_{w_2}\subset \Psi_w\subset\Delta(\mk r_\xi^-)$. In particular, all $w_j$ belong necessarily to ${^\xi W}$. It follows that the $R_\xi^-$-stabilizers of the points $x_{w_j}$ are nested, i.e.,
$$
1 = (R_\xi^-)_{x_{1}} \subset (R_\xi^-)_{x_{w_l}} \subset ... \subset (R_\xi^-)_{x_{w}}
$$
The pair $(\xi,w_j)$ is fit if and only if the generic $\hat R_\xi^-$-stabilizer on $L_\xi x_{w_j}$ is trivial. Recall that $\hat R_\xi^-\subset R_\xi^-$, and also that the $R_\xi^-$-stabilizers on $L_\xi x_{w}$ are $L_\xi$-conjugate, so $L_\xi$-conjugate to $(R_\xi^-)_{x_w}$. Thus the above chain of inclusions implies that, on each $L_\xi x_{w_j}$, the generic $\hat R_\xi^-$-stabilizer is trivial. Hence $w_{j}\in {^\xi W}_{\rm fit}(l-j+1)$. This proves the first statement.

For the second statement of the lemma, let us use the notation from the proof of Lemma \ref{Lemma InvertedSets} and note that
$$
s_{\beta} w \sigma_\xi\lambda = -(positive\; number)\beta (\xi) + w\sigma_\xi\lambda(\xi) > w\sigma_\xi\lambda(\xi) \;.
$$
\end{proof}

\begin{thm}\label{Theo Ck}
For $k\geq1$, $C_k^{\hat G}(X)\subset \Lambda_\R^+$ is a rational polyhedral cone, and can be expressed in the following ways:
\begin{align*}
C_k^{\hat G}(X) & = \ol{\{\lambda\in\Lambda_\Q^{++}: {\rm codim}_X X^{us}(\lambda)\geq k \}} \\
 & = \{\lambda\in\Lambda_\R^+:\; \lambda(\sigma_j^{-1}w^{-1}\xi_j)\leq 0 \;,\; w\in {^{\xi_j}W}_{\rm fit}(r_j-\hat r_j-k+1) , j=1,...,q \}\;.
\end{align*}
In the relative topology of $\Lambda_\R^+$, we have $C_{k+1}^{\hat G}(X)\subset {\rm Int}\, C_k^{\hat G}(X)$.

In particular, the $\hat G$-ample and $\hat G$-movable cones are obtained for $k=1,2$ as
\begin{align*}
& C^{\hat G}(X) = \{\lambda\in\Lambda_\R^+: {^{\xi_j}W}_{\rm fit}(r_j-\hat r_j)\subset W^{0-}(\sigma_j\lambda,\xi_j) \;,\; \forall j \} \;,\\
& {\rm Mov}^{\hat G}(X) = \{\lambda\in\Lambda_\R^+ : {^{\xi_j}W}_{\rm fit}(r_j-\hat r_j+1)\subset W^{0-}(\sigma_j\lambda,\xi_j) \;,\;\forall j\} \;.
\end{align*}
\end{thm}
 
\begin{proof}
In our description of the Kirwan-Ness stratification of the unstable locus in Theorem \ref{Theo KirwanStratFlag}, for any given $\lambda\in\Lambda^{++}$, we have described the set of stratifying pairs in terms of the set $\mk S\mk W_\lambda$, which is a subset of the set $\Xi\mk W^+_\lambda$ (see Def. \ref{Def OPSforLeviOrbits XiWlambda}) parameterizing the unstable parabolic orbits defined by dominant OPS of $\hat G$. We have
$$
X^{us}(\lambda) = \bigcup\limits_{(\xi,w)\in\Xi\mk W^+_\lambda} \ol{\hat R_\xi^- P_\xi x_{w\sigma_\xi}}
$$
The dimension formula for the strata derived from the general stratification Theorem \ref{Theo Kirwan-Ness} implies that all stratifying pairs $(\xi,w)\in \mk S\mk W_\lambda$ are fit. Since the maximum dimension is attained at a stratifying element, we deduce that
$$
{\rm codim}_X X^{us}(\lambda) = \min\{ r_\xi-\hat r_\xi - l(w) : (\xi,w)\in \Xi\mk W_\lambda^+ \cap \Xi\mk W_{\rm fit} \} \;.
$$
Then the cone $C_k^{\hat G}(X)$ can be characterized as those $\lambda$ whose unstable locus does not contain any $(k-1)$-codimensional parabolic orbit $P_\xi x_{w\sigma_\xi}$ corresponding to a fit pair $(\xi,w)$. This means $w\sigma_\xi\lambda(\xi)\leq 0$ for all fit $(\xi,w)$ with $w\in {^\xi W}(r_\xi-\hat r_\xi-k+1)$. This proves description of $C_k^{\hat G}(X)$ given in the theorem. 

To prove the second statement, we shall use the following lemma.

\begin{lemma}\label{Lemma codim Jump 1}
Suppose that $C_1, C_2\subset \Lambda_\R^+$ are $\hat T$-GIT-classes intersecting the interior of the Weyl chamber.

(a) If $\ol{C_1}\supset C_2$, then
$$
0\leq {\rm codim}_X X^{us}(C_2) - {\rm codim}_X X^{us}(C_1)  \leq 1 \;.
$$

(b) If $C_1, C_2$ are GIT-chambers sharing a facet $C_{12}$, then
$$
| {\rm codim}_X X^{us}(C_1) - {\rm codim}_X X^{us}(C_2) | \leq 1 \;.
$$
\end{lemma}

\begin{proof}
Suppose $\ol{C_1}\supset C_2$. Then we have $X^{us}(C_1)\supset X^{us}(C_2)$. In particular, ${\rm codim}_X X^{us}(C_1) \leq {\rm codim}_X X^{us}(C_2)$. 

The change in the unstable locus $X^{us}(\lambda)$ as $\lambda$ passes from $C_1$ to $C_2$ is necessarily reflected in a change of the set of stratifying 
elements $\mk S\mk W_\lambda$. By the continuity of the bilinear pairing between $\Lambda_\R$ and $\Gamma_\R$, there exists a stratifying 
pair $(\xi,w)$, with $w\in {^\xi W}^+(C_1,\xi) \cap {^\xi W}^0(C_2,\xi)$. For any such $(\xi,w)$, we have $(\xi,w)\in\Xi\mk W_{\rm fit}$ and we 
can apply Lemma \ref{Lemma SeqFit}. The element $w_2$ produced by that lemma is fit for $\xi$, satisfies $w_2\in {^\xi W}^+(C_2,\xi)$ and has length $l(w)-1$. Thus
$$
\hat GP_\xi x_{w_2\sigma_\xi} \subset X^{us}(C_2) \quad,\quad {\rm codim}_X \hat GP_\xi x_{w_2\sigma_\xi} = {\rm codim}_X \hat GP_\xi x_{w\sigma_\xi} +1 \;.
$$
Therefore, the inequality
$$
 {\rm codim}_X \hat G X_\xi^{us}(C_2) - {\rm codim}_X \hat G X_\xi^{us}(C_1)  \leq 1
$$
holds for any $\xi$. Hence it also holds for the codimensions of the entire unstable loci. This proves part (a).

For part (b) a similar argument works. If the unstable locus changes for some $\xi\in\Xi_{\rm max}$, then there exists $w\in {^\xi W}^+(C_1,\xi) \cap {^\xi W}^0(C_{12},\xi) \cap {^\xi W}^-(C_2,\xi)$. Then we have $w_2\in {^\xi W}^+(C_{12},\xi)$ and may we proceed as above.
\end{proof}

Form the lemma we deduce that the regular faces of $C_k^{\hat G}(X)$ give $k$-dimensional unstable loci, since they are contained in the closure of some GIT-chambers from $C_{k-1}^{\hat G}(X)$. This implies that the chambers contained in $C_k^{\hat G}(X)$, whose closure intersect the regular boundary of $C_k^{\hat G}(X)$, must also have $k$-dimensional unstable loci. These chambers form then a ``layer'' isolating $C_{k+1}^{\hat G}(X)$ from the regular boundary of $C_k^{\hat G}(X)$. This completes the proof of the theorem.
\end{proof}

\begin{example} It is not hard to show that for any $SL_2$-subgroup of $SL_3$ one has $C^{\hat G}(X)=\Lambda_\R^+$ and ${\rm Mov}^{\hat G}(X)=\emptyset$. Indeed, there are two conjugacy classes, but their Cartan subalgebras coincide, up to conjugacy, as vector spaces $\hat{\mk t}\subset \mk t$ (endowed however with different lattices $\hat\Gamma$). We have $\hat{\mk t}=\R\rho^{\vee}$, and it suffices to evaluate weights on $\xi=\rho^\vee=\alpha_1^\vee+\alpha_2^\vee$, which is regular, so $\mk r_\xi=\mk n=3$, $\hat r_\xi=1$. Since $\hat L_\xi\cong\C^*$ and $\hat L'_\xi=1$, all pairs $(\xi,w)\in\Xi\mk W_\lambda^+$ are stratifying. A simple computation yields
$$
\Lambda^{++} = \{\lambda\in\Lambda: W^+(\lambda,\xi) = W(l\leq 1)=\{1,s_1,s_2\} \} \;,\; X^{us}(\lambda)=\hat G B_{x_{s_1}}\cup \hat GB_{x_{s_2}} \;.
$$
Thus the facets of $C^{\hat G}(X)$ constructed by our theorem coincide with the walls of the Weyl chamber, and the interior constitutes a single GIT-chamber with unstable locus of codimension 1.
\end{example}

Since we shall be interested in $\hat G$-movable chambers, we record the following immediate corollary.

\begin{cor}\label{Cor Sufficient for movable chambers}
If $C_k^{\hat G}(X)\ne \emptyset$, then $X$ admits GIT-chambers where the unstable locus has codimension $k-1$. In particular, if there exists $\lambda\in\Lambda^{++}$ with ${\rm codim}_X X^{us}(\lambda)>2$, then $X$ admits $\hat G$-movable chambers. 
\end{cor}

\begin{rem}
In our previous work, \cite{Seppa-Tsa-Principal}, we have considered the case where $\hat G$ if a principal $SL_2$-subgroup $G$. We have shown that $\hat G$-movable chambers exist, except for a small number of degenerate cases for $G$. Under some more assumptions, e.g. $G$ not having simple factors of rank 1 or 2, the entire ample cone is $\hat G$-movable.
\end{rem}

\begin{rem}
There is an obvious upper bound for the codimension of the unstable locus, which can be deduced from the above theorem:
$$
\max\{{\rm codim}_X X^{us}(\lambda) : \lambda\in\Lambda^{++}\}  \leq \min\{r_\xi-\hat r_\xi: \xi\in\Xi_{\rm max} \} \;.
$$
In particular, we have $C^{\hat G}(X)=0$, when this upper bound is zero. This holds for instance for the natural embedding $\hat G=Sp_{2n}\subset SL_{2n}=G$, as well as in in case $\hat{\mk g}$ contains simple ideals of $\mk g$. 
\end{rem}

\begin{example}\label{Examp rho diag}
Let us consider the case $\lambda=\rho=\frac12 \sum\limits_{\alpha\in\Delta^+} \alpha$, the smallest strictly dominant weight. We shall estimate the codimension of $X^{us}(\rho)$ in terms of invariants of the embedding $\hat G\subset G$ and thus give a criterion for existence of $\hat G$movable chambers. We shall give some more precise calculations for diagonal embeddings $\hat G \subset \hat G^{\times k}=G$, where $\rho = (\hat\rho,...,\hat \rho)$, specifically for $\hat G=SL_{m}$.

Let us begin with the general remark, that for $w\in W$, we have $w\rho = \frac12 (\langle \Phi_{w^{-1}}^c \rangle - \langle \Phi_{w^{-1}} \rangle)$, where $\langle\Phi\rangle$ denotes the sum of the elements of any subset $\Phi\in\Lambda$. Evaluated at any $\xi\in\Gamma^+$ this gives
$$
w\rho(\xi) = \frac12 (\langle \Phi_{w^{-1}}^c \rangle - \langle \Phi_{w^{-1}} \rangle)(\xi) = \frac12 (\langle \Phi_{w^{-1}}^c\cap \Delta(\mk{r}_\xi) \rangle - \langle \Phi_{w^{-1}}\cap \Delta(\mk{r}_\xi) \rangle))(\xi) \;.
$$
Since $\Phi_{w^{-1}}^c=\Phi_{w_0w^{-1}}$, we conclude that either $w\rho(\xi)=ww_0\rho(\xi)=0$, or exactly one of $w$ and $ww_0$ 
belongs to $W^+(\rho,\xi)$ while the other one belongs to $W^-(\rho,\xi)$. Also $w\in {^\xi W}$ if and only if $\Phi_{w^{-1}}\subset \Delta(\mk r_\xi)$. 
Put 
\begin{align}
a_\xi=\min\{\alpha(\xi):\alpha\in\Delta(\mk r_\xi)\} \quad b_\xi=\max\{\alpha(\xi):\alpha\in\Delta(\mk r_\xi)\}. \label{E: rootvalOPS}
\end{align}
Then,
\begin{align*}
 a_\xi(r_\xi-l(w)) - b_\xi l(w) \leq 2w\rho(\xi) \leq b_\xi (r_\xi - l(w)) - a_\xi l(w)  \;.   
\end{align*}
It follows that, for $w\in {^\xi W}$,
\begin{align*}
& l(w) < \frac{a_\xi r_\xi}{a_\xi+b_\xi} \;\Lw\; w\in {^\xi W}^+(\rho,\xi) \;;\\
& l(w) \geq \frac{b_\xi r_\xi}{a_\xi+b_\xi} \;\Lw\; w\in {^\xi W}^{0-}(\rho,\xi)
\end{align*}
Hence
$$
\frac{a_\xi r_\xi}{a_\xi+b_\xi} - 1  \leq l_{\xi,\rho}^+ < \frac{b_\xi r_\xi}{a_\xi+b_\xi} 
$$
and
$$
{\rm codim}_X X^{us}(\rho) \geq  \min\limits_{\xi\in\Xi_{\rm max}}\{r_\xi-\hat r_\xi - l_{\xi,\rho}^+ \} > \min\limits_{\xi\in\Xi_{\rm max}} \{\frac{a_\xi}{a_\xi+b_\xi}  r_\xi - \hat r_\xi \} \;.
$$
\end{example}

We can use this concrete case, where the codimension is expressed in terms of structural invariants of the embedding $\hat G\subset G$, to obtain the following general criterion for existence of $\hat G$-movable chambers.

\begin{prop}
Given an embedding $\hat G\subset G$, if $\min\limits_{\xi\in\Xi_{\rm max}} \{\frac{a_\xi}{a_\xi+b_\xi}  r_\xi - \hat r_\xi \}\geq 2$ (cf. \ref{E: rootvalOPS}), 
then the $\hat G$-ample cone on $X$ admits $\hat G$-movable chambers.
\end{prop}

Let us consider now a diagonal embedding $\hat G\subset\hat G^{\times k}=G$, Any Weyl chamber of $\hat G$ is contained, as a diagonal, in a Weyl chamber of $\hat G$. The (maximal) Levi subgroups of $G$ defined by nonzero elements of $\hat\Gamma^+$ are the $k$-fold products of (maximal) Levi subgroups of $\hat G$. We have $r_\xi=k\hat r_\xi$. The OPS defining the maximal Levi subgroups are the fundamental coweights $\Xi_{\rm max}=\{\hat\xi_1,...,\hat\xi_{\hat\ell}\}$. Furthermore, for $\xi_j\in\Xi_{\rm max}$ we have $\hat a_{\xi_j}=\hat a_{\xi_j}=2$, $b_{\xi_j}=\hat b_{\xi_j} = 2m_j$, where $m_j$ is the $j$-th coefficient of the highest root of $\hat G$ expressed as a sum of simple roots, i.e., $\tilde{\alpha}=\sum m_j \hat\alpha_j$. Hence
$$
{\rm codim}_X X_{\xi_j}^{us}(\rho) > \frac{1}{1+m_j}  k\hat r_j - \hat r_j \;.
$$
We can also see that the codimension of the unstable locus tends to $\infty$ when $k\to\infty$. Concerning $\hat G$-movable chambers, one can easily calculate that
\begin{align*}
k\geq \max\{ \frac{\hat r_j+2}{\hat r_j} (1+m_j) : j=1,...,\hat\ell\}  & \;\Lw\; {\rm codim}_X X^{us}(\rho)>2 \\
 & \;\Lw\; \exists \;\hat G-movable \; chambers\;.
\end{align*}

In particular, one can deduce the following.

\begin{prop}
If $\hat G$ is a product of classical groups and $\hat G\subset G=\hat G^{\times k}$ is a diagonal embedding with $k\geq 5$, then the $\hat G$-ample cone admits $\hat G$-movable chambers.
\end{prop}

\begin{example}
Let us consider the case $\hat G=SL_{\hat\ell+1}$, where $m_j=1$ for all $j$. Then the above bound means that there are $\hat G$-movable chambers for $k> 2\frac{\hat\ell+2}{\hat\ell}$. Let us Going back a few steps, we compute, $w\in {^{\xi_j}W}$ we have
$$
w\rho(\xi_j) = \frac12 (\langle \Phi_{w^{-1}}^c \rangle - \langle \Phi_{w^{-1}} \rangle)(\xi_j) = \frac12 (r_j-l(w) - l(w) )=\frac12 r_j - l(w) = \frac{k}{2}\hat r_j - l(w) \;.
$$
Thus $l_{j,\lambda}^+ = \lfloor \frac{r_j-1}{2} \rfloor$.
$$
{\rm codim}_X\hat G X^{us}_{\xi_j}(\rho) \geq r_j-\hat r_j - l_{j,\lambda}^+ = \lceil \frac{r_j+1}{2} \rceil - \hat r_j = \lceil \frac{k\hat r_j+1}{2} \rceil - \hat r_j = \lceil \frac{(k-2)\hat r_j + 1}{2} \rceil \;.
$$
The minimum value over $j=1,...,\hat\ell$ is attained at $j=1$, where $\hat r_j=\hat\ell$ and
$$
{\rm codim}_X X^{us}(\rho) = {\rm codim}_X \hat G X^{us}_{\xi_1}(\rho) \geq \lceil \frac{(k-2)\hat\ell + 1}{2} \rceil \;.
$$
We obtain ${\rm codim}_X X^{us}(\rho) >2$ except in the following cases:
\begin{gather*}
{\rm codim}_X \hat G X^{us}(\rho) = \begin{cases} 1 \;,\; if\quad k=2 \\ 1 \;,\; if\quad k=3 , \hat\ell=1 \\ 2 \;,\; if\quad k=3 , \hat\ell=2,3 \\ 2 \;,\; if\quad k=4 , \hat\ell=1 \end{cases}
\end{gather*}
In particular, in all cases except the above, $\hat G$-movable chambers do exist.
\end{example}

\section{Popov's tree-algorithm}\label{Section Popov}

Here we present an algorithm allowing to determine whether a given $\lambda\in\Lambda^{++}$ belongs to $C^{\hat G}(X)$ or not. Having in mind our description of Kirwan stratification of $X^{us}(\lambda)$, where the non-emptiness of the proposed strata depends on whether certain $W$-translate $w\sigma_\xi\lambda$ belongs to $C^{\hat L'_\xi}(L_\xi x_w)$, this algorithm can be applied to determine the entire stratification. The idea is due to Popov, \cite{Popov-Nullforms}, who developed the method in his study of unstable points in a linear representation space of a reductive group, the classical nullcone of a representation.
There is a common generalization of his and our settings, where $X=G/P$ is a partial flag variety, $\mP(V)=SL(V)/P_1$ in the classical case, with an action of a reductive subgroup $\hat G\subset G$. One particular feature of complete flag varieties, as well as of projective spaces, is that the closed orbits of Levi subgroups $L\subset G$ are of the same type, i.e., complete flag varieties, or, respectively, projective spaces. This is important, since the algorithm uses recursion, whose step refers to Levi subgroups acting on their closed orbits in $X$. This latter fact remains somewhat hidden in the classical case, where one considers linear subspaces of a vector space without necessarily mentioning Levi subgroups of its linear group.

Let $\mk Z_X=\{L_\xi x_w:\xi\in\hat\Gamma,w\in W\}$ denote the set of closed orbits in $X$ of Levi subgroups of $\hat G$ defined by one-parameter subgroups of $\hat T$. Next we define a rooted tree $\mc T_\lambda = T_{\hat G, X,\lambda}$, whose vertices are associated to elements of $Z$ with a natural orientation and signature, which allows to determine, by a recursive algorithm, whether $\lambda$ defines a $\hat G$-ample line bundle on $X$ or not. 

Every rooted tree is endowed with a natural orientation of the edges, pointing from the root to its adjacent vertices, and defined inductively for the rest of the tree.

\begin{defin}
Let $\lambda\in \Lambda^{++}$. We denote
$$
\mk M_X=\mk M_{\hat G, X,\lambda} = \{L_\xi x_{w\sigma_\xi}\in \mk Z_X: \xi=\xi_{L_\xi,w,\sigma_\xi\lambda}\in\hat\Gamma^+\setminus\{0\}, w\in {^\xi W}, l_\xi(w) = r_\xi-\hat r_\xi\} \;.
$$
Analogously, for any $Z=L_\xi x_w\in \mk Z_X$, endowed with the action of $\hat L'_\xi$ and the line bundle given by $w\sigma_\xi\lambda$, we denote
$$
\mk M_Z=\mk M_{\hat L'_\xi, Z,w\sigma_\xi\lambda} \;.
$$
We define a rooted tree $\mc T_\lambda$ with vertices $a_{(Z_j)}$ associated to sequences of nested elements $(Z_j)=(Z_0\supset Z_1\supset\dots\supset Z_p)$ of $\mk Z_X$, starting at $Z_0=X$. and satisfying $Z_{j+1}\in \mk M_{Z_j}$. The root of $\mc T_\lambda$ is $a_{(X)}$. The vertices adjacent to $a_{(X)}$ are $a_{X\subset Z}$ for $Z\in \mk M_X$. The vertices stemming from $a_{X\supset Z_1\supset\dots\supset Z_p\supset Z}$ are, by definition, $a_{X\supset Z_1\supset\dots\supset Z_p\supset Z}$ for $Z\in \mk M_{Z_p}$.

The height of a vertex $a$ is defined as the maximum length of an oriented path in $\mc T_\lambda$ starting at $a$.

A signature on the tree $\mc T_\lambda$ is defined as follows: a vertex $a$ is given a sign ``$-$'' if there exists an arrow in $\mc T_\lambda$ ending at a vertex $b$ with $sign(b)=+$; otherwise, $a$ is given a sign ``$+$''.
\end{defin}

\begin{rem}
(i) The vertices of height $0$, called the leaves, always have sign ``+''.

(ii) If $\hat G$ is abelian, then the tree associated to any $\lambda\in\Lambda^{++}$ consists only of the root, $\mc T_\lambda=\{a_{(X)}\}$. Hence the sign is always ``+'', which corresponds to the fact that $C^{\hat G}(X)=\Lambda^{+}_\R$. 

(iii) The maximal height of a vertex in the tree $\mc T_\lambda$ is the height of the root. It does not exceed ${\rm rank}(\hat G)$, since for chains $X\supset...\supset Z$ of that length, or higher, the semisimple part of the Levi subgroup $\hat L\subset \hat G$ preserving $Z$ is abelian.
\end{rem}

\begin{thm}
Let $\lambda\in \Lambda^{++}$. The line bundle $\mc L_\lambda$ on $X$ is $\hat G$-ample if and only if the root of $\mc T_\lambda$ has sign plus, i.e.,
$$
C^{\hat G}(X)\cap\Lambda^{++} = \{\lambda\in\Lambda^++ : sign(a_{(X)})=+ \;\;{\rm in}\;\;\mc T_\lambda \} \;.
$$
\end{thm}

\begin{proof}
We follow the idea of Popov, \cite{Popov-Nullforms}. Let us remark that the branches of $\mc T_\lambda$ are again trees of the same type. More precisely, let $a=a_{(Z_j)}$ be a vertex in $\mc T_\lambda$ and let $Z=Z_p$ be the last variety in the sequence defining $a$. Let $(\xi,w)\in\Xi_\lambda^+\times W$ be the elements associated to $Z,\lambda$ according to the above definition. Then the branch of $\mc T_\lambda$ starting at $a$ is identical with the tree $\mc T_{\hat L'_\xi,Z,w\sigma_\xi\lambda}$. This tree depends only on $Z$ and $\lambda$, but not on the sequence $(Z_j)$ connecting $X$ to $Z$; we shall denote it by $\mc T_\lambda(Z)$.

We shall prove the theorem by induction on the height of the whole tree, i.e., the height of the root. In the above remark, we noticed that the vertices of height $0$ always have sign plus. The height of the root $a_{(X)}$ is $0$ if and only if $\mk M_X=\emptyset$. The latter implies, via Theorem \ref{Theo KirwanStratFlag}, that there are no Kirwan-Ness strata in $X^{us}(\lambda)$ of codimension $0$, hence $\lambda\in C^{\hat G}(X)$. Thus the statement of the theorem holds in the base case. Assume it holds for trees of height one less than the height of $\mc T_\lambda$. The sign of the root $a_{(X)}$ is minus if and only if there is an adjacent vertex $a_{X\supset Z}$ with $Z\in\mk M_X$ and sign plus. This means that the root of $\mc T_{\lambda}(Z)$ has sign plus.
By hypothesis, this is equivalent to $w\sigma_\xi\lambda\in C^{\hat L'_\xi}(Z)$, In such a case $(\xi,w)$ is a stratifying pair for $X_{\hat G}^{us}(\lambda)$ and, since $Z\in \mk M_X$, we have ${\rm codim}_X \hat G P_{\xi}x_{w\sigma_\xi}=0$. This is in turn equivalent to $\lambda\notin C^{\hat G}(X)$.
\end{proof}

\begin{example} It is not hard to show that, for $\hat G$ of rank 1 or 2, a given $\lambda\in\Lambda^{++}$ belongs to $C^{\hat G}(X)$ if and only if $\mc T_\lambda$ does not have branches of length 1. For ${\rm rank}(\hat G)=1$, this means $\mc T_\lambda=\{a_{(X)}\}$.
\end{example}

\section{Mori chambers}\label{Sect Mori chambers}

The goal of this section is to prove theorems \ref{T: Mori-GITchambers} and \ref{T: conesineff} concerning the structure of the effective cone of any 
GIT quotient, $Y$, of $X$ defined by a $\hat G$-movable chamber $C$ in the $\hat G$-ample cone $C^{\hat G}(X)$. By a result of \cite{Seppanen-GlobBranch}, 
such a quotient is a Mori dream space, whose pseudoeffective cone is naturally identified with the $\hat G$-ample cone $C^{\hat G}(X)$. Here we study the 
birational geometry $Y$ and show that the Mori chambers of $\ol{\rm Eff}(Y)$ correspond to the GIT-chambers of $C^{\hat{G}}(X)$. \\

We first recall the notion of Mori equivalence for divisors: two big divisors, $D$ and $D'$, on a projective variety $Y$, with finitely generated section rings 
$R(Y, \mathcal{O}_Y(D))$ and $R(Y', \mathcal{O}_{Y}(D'))$ and natural evaluation maps 
$f_D: Y \dasharrow \rm{Proj}(R(Y, \mathcal{O}_Y(D)))$ and $f_{D'}: Y \dasharrow \rm{Proj}(R(Y, \mathcal{O}_Y(D')))$, are Mori equivalent if there is an isomorphism 
$$\varphi: Y_D:=\rm{Proj}(R(Y, \mathcal{O}_Y(D))) \to \rm{Proj}(R(Y, \mathcal{O}_Y(D')))=:Y_{D'}$$ making the following diagram commute: 
\begin{align*}
  \xymatrix{
 Y \ar@{-->}[dr]^{f_{D'}}\ar@{-->}[r]^{f_D} &  Y_D \ar[d]^{\varphi} \\
  & Y_{D'}
 }
\end{align*}
(cf. \cite{hk}). A Mori chamber in the pseudoeffective cone $\overline{\rm{Eff}}(Y)$ is the closure of a full dimensional Mori equivalence class.\\  

We now assume that $Y=Y_{\lambda_0}=X^{ss}(\lambda_0)//\hat{G}$, with projection morphism 
$$\pi: X^{ss}(\lambda_0)  \to X^{ss}(\lambda_0)//\hat G,$$ is a quotient such that $\lambda_0$ belongs to a $\hat G$-movable 
chamber in $C^{\hat {G}}(X)$.
 
If $\lambda \in C^{\hat G}(X)$ is a strictly dominant weight for which the line bundle $\mathcal{L}_\lambda$ on $X$ descends to a line bundle $L_\lambda$ on $Y$, the section 
ring $R(Y, L_\lambda)$ of $L_\lambda$ is finitely generated, namely $R(Y, L_\lambda) \cong R(X, \mathcal{L}_\lambda)^{\hat G}$, where the latter ring is 
finitely generated by the theorem by Hilbert and Nagata. Evaluating homogeneous elements of $R(Y, L_\lambda)$ in points in $Y$ outside the stable base 
locus $\mathbb{B}(L_\lambda)$ of $L_\lambda$ then yields a rational map 
\begin{align*}
 f_ \lambda: Y \dasharrow Y_\lambda={\rm Proj}(R(Y, L_\lambda)),
 \quad f_\lambda(y):={\rm ker}\,\,{\rm ev}_{y}, \quad y \in Y \setminus \mathbb{B}(L_\lambda),
\end{align*}
where $${\rm ev}_y:=\oplus_{k=1}^ \infty {\rm ev}_{y,k},$$
and
\begin{align*}
{\rm ev}_{y, k}(s):=s(y) \in (L_{\lambda}^k)_{y}/\mathfrak{m}_y(L_\lambda^k)_y, \quad s \in H^0(Y, L_\lambda^k),
\end{align*}
where $\mathfrak{m}_y$ denotes the maximal 
ideal in the stalk $\mathcal{O}_{Y, y}$ of the structure sheaf $\mathcal{O}_Y$ of $Y$. 

\begin{lemma} \label{L: ratmapgit}
 The rational map $f_\lambda$ is induced by GIT, that is, it is the map 
 $$\pi(X^{ss}(\lambda_0) \cap X^{ss}(\lambda)) \to X^{ss}(\lambda)//\hat G=Y_\lambda$$
 induced on quotients by the inclusion $X^{ss}(\lambda_0) \cap X^{ss}(\lambda) \hookrightarrow X^{ss}(\lambda)$.
\end{lemma}

\begin{proof}
 Since $\mathcal{L}_\lambda$ is very ample, we can write $X$ as $X={\rm Proj}(R(X, \mathcal{L}_\lambda))$. On the other hand, $Y_\lambda$ is given by 
 $Y_\lambda={\rm Proj}(R(X, \mathcal{L}_\lambda)^{\hat G})$. Now, the inclusion 
 $R(X, \mathcal{L}_\lambda)^{\hat G} \hookrightarrow R(X, \mathcal{L}_\lambda)$ yields a rational 
 map of projective spectra $$q_\lambda: {\rm Proj}(R(X, \mathcal{L}_\lambda))\dasharrow {\rm Proj}(R(X, \mathcal{L}_\lambda)^{\hat G}),$$ 
 given on the level of points by 
 \begin{align*}
  q_\lambda(\mathfrak{p}):=\mathfrak{p} \cap R(X, \mathcal{L}_\lambda)^{\hat G}, \quad \mathfrak{p} \in U, 
 \end{align*}
where $U$ is the set of all homogeneous relevant prime ideal for which the homogeneous prime ideals $\mathfrak{p} \cap R(X, \mathcal{L}_\lambda)^{\hat G}$ 
is relevant, i.e., does not contain $H^0(X, \mathcal{L}_\lambda^k)^{\hat G}$ for all positive integers $k$. The closed points of $U$ are then precisely the 
points in the semistable locus $X^{ss}_{\hat G}(\lambda)$. Clearly, $q_\lambda$ is $\hat{G}$-invariant, and we claim that in fact $q_\lambda=\pi_\lambda$. 
Before proving this claim, we show that the claim of the lemma follows from the identity $q_\lambda=\pi_\lambda$. 
Indeed, we can lift $f_\lambda$ to an evaluation map 
$\pi^*f_\lambda: X^{ss}_{\hat G}(\lambda_0) \cap X^{ss}_{\hat G}(\lambda) \to {\rm Proj}(R(X, \mathcal{L}_\lambda)^{\hat G})$ given by  
\begin{align*}
 x \mapsto ({\rm ker} \,\,{\rm ev}_x) \cap R(X, \mathcal{L}_\lambda), \quad x \in X^{ss}_{\hat G}(\lambda_0) \cap X^{ss}_{\hat G}(\lambda).
\end{align*}
If $F_\lambda: X \to {\rm Proj}(R(X, \mathcal{L}_\lambda))$ denotes the natural map
\begin{align*}
 F_\lambda(x):={\rm ker} \,\, {\rm ev}_x, \quad x \in X,
\end{align*}
where the evaluation maps ${\rm ev}_x$, for $x \in X$, are defined as above, but for the line bundle $\mathcal{L}_\lambda$ on $X$,  
the map $\pi^*f_ \lambda$ can be written as the composition 
\begin{align}
\pi^*f_\lambda=q_\lambda \circ F_\lambda \mid_{X^{ss}_{\hat G}(\lambda_0) \cap X^{ss}_{\hat G}(\lambda)}. \label{E: pbmorimap}
\end{align}
Now, since $X \cong {\rm Proj}(R(X, \mathcal{L}_\lambda))$, the morphism $F_\lambda$ providing an isomorphism, we can in fact identify $F_\lambda$ 
with the identity morphism of $X$. Using this identification, the identity \eqref{E: pbmorimap} in fact says that $\pi^*f_\lambda$ is given as the 
composition of the quotient morphism $q_\lambda=\pi_\lambda: X^{ss}_{\hat G}(\lambda) \to X^{ss}_{\hat G}(\lambda)//\hat{G}$ with the 
inclusion of the open subsets $X^{ss}_{\hat G}(\lambda_0) \cap X^{ss}_{\hat G}(\lambda) \hookrightarrow X^{ss}_{\hat G}(\lambda)$, and this is indeed the 
claim of the lemma.

We then conclude the proof by showing that $q_\lambda=\pi_\lambda$. For this, it suffices to show that $q_\lambda$ and $\pi_\lambda$ coincide on open affine subsets 
defining an open affine $\hat{G}$-invariant covering of $U$. Let therefore 
$s_1,\ldots, s_m \in H^0(X, \mathcal{L}_\lambda)^{\hat G}=H^0(Y, L_\lambda)$, for some $m \in \N$, be homogeneous generators of the invariant ring 
$R(X, \mathcal{L}_\lambda)^{\hat G}=R(Y, L_\lambda)$. (By replacing $\mathcal{L}_\lambda$ by a power, if necessary, we may without loss of generality assume that this invariant ring has generators in degree one.) Then, putting 
\begin{align*}
 X_{(s_i)}&:=\{\mathfrak{p} \in {\rm Proj}(R(X, \mathcal{L}_\lambda)): s_i \notin \mathfrak{p}\} \subseteq X,\\
 Y_{(s_i)}&:=\{\mathfrak{p} \in {\rm Proj}(R(Y, L_\lambda)): s_i \notin \mathfrak{p}\} \subseteq Y,
 \end{align*}
 for $i=1,\ldots, m$,
and recalling that these open subsets are affine, namely
\begin{align*}
 X_{(s_i)} \cong {\rm Spec}(R(X, \mathcal{L}_\lambda)_{(s_i)}), \quad Y_{(s_i)} \cong {\rm Spec}(R(Y, L_\lambda)_{(s_i)}),  
\end{align*}
where 
\begin{align*}
 R(X, \mathcal{L}_\lambda)_{(s_i)}&=\left\{\frac{s}{s_i^k} : k \in \N, \quad s \in H^0(X, \mathcal{L}_\lambda^k)\right\}, \\
 R(Y, L_\lambda)_{(s_i)}&=\left\{\frac{s}{s_i^k} : k \in \N, \quad s \in H^0(Y, L_\lambda^k)\right\}  
\end{align*}
are the homogeneous localizations of the respective rings with respect to the degree-one element $s_i$. 

The action of $\hat{G}$ on $R(X, \mathcal{L}_\lambda)$ by graded ring automorphisms induces an action on the homogeneous localization $R(X, \mathcal{L}_\lambda)_{(s_i)}$ 
given by 
\begin{align*}
 g(\frac{s}{s_i^k}):=\frac{g(s)}{s_i^k}, \quad g \in \hat{G}, \quad s \in H^0(X, \mathcal{L}_\lambda^k), \quad k \in \N,
\end{align*}
and for this action we clearly have
\begin{align*}
 (R(X, \mathcal{L}_\lambda)_{(s_i)})^{\hat G}=(R(X, \mathcal{L}_\lambda)^{\hat G})_{(s_i)}=R(Y, L_\lambda)_{(s_i)},
\end{align*}
that is, the operation of taking $\hat{G}$-invariants commutes with homogeneous localization with respect to $s_i$. 
In other words, the inclusion 
\begin{align}
R(X, \mathcal{L}_\lambda)^{\hat G}_{(s_i)} \hookrightarrow R(X, \mathcal{L}_\lambda)_{(s_i)} \label{E: lochilbquot}
\end{align}
of the subring of $\hat{G}$-invariants is given by the 
restriction $q_\lambda\mid_{X_{(s_i)}}: X_{(s_i)} \to Y_{(s_i)}$. Since the embedding of rings \eqref{E: lochilbquot} defines a Hilbert quotient, 
the uniqueness of good quotients implies that it coincides with the restriction of $\pi_\lambda$ to 
$X_{(s_i)}$. This shows that $q_\lambda=\pi_\lambda$.
\end{proof}

\begin{lemma} \label{L: sscodim1}
 Let $\lambda, \lambda' \in C^{\hat{G}}(X)$ be Mori equivalent dominant weights each belonging to some GIT-chamber.
 Then the semistable loci $X^{ss}(\lambda)$ and $X^{ss}(\lambda')$ are equal in codimension one, that is, they coincide outside a 
 closed subset of $X$ of codimension at least two.
\end{lemma}

\begin{proof}
Let $f_\lambda: Y \dasharrow Y_{\lambda}$ and $f_{\lambda'}: Y \dasharrow Y_{\lambda'}$ be the rational maps defined by the line bundles on $Y$. 
By assumption, there is an isomorphism $\varphi: Y_{\lambda} \to Y_{\lambda'}$ yielding a commutative diagram
\begin{align*}
 \xymatrix{
 Y \ar@{-->}[dr]^{f_{\lambda'}}\ar@{-->}[r]^{f_\lambda} & Y_\lambda \ar[d]^{\varphi} \\
  & Y_{\lambda'},
 }
\end{align*}
from which it follows that ${\rm exc}(f_\lambda)={\rm exc}(f_{\lambda'})$. 
Since $\lambda_0, \lambda, \lambda'$ are all in GIT-chambers, $f_\lambda$ and $f_{\lambda'}$ define isomorphisms 
\begin{align*}
 \pi(X^{ss}(\lambda_0) \cap X^{ss}(\lambda)) 
 \stackrel{f_\lambda}{\cong} \pi_\lambda(X^{ss}(\lambda_0) \cap X^{ss}(\lambda)) \subseteq Y_\lambda,\\
 \pi(X^{ss}(\lambda_0) \cap X^{ss}(\lambda')) 
 \stackrel{f_\lambda'}{\cong} \pi_\lambda(X^{ss}(\lambda_0) \cap X^{ss}(\lambda')) \subseteq Y_\lambda'.    
\end{align*}
Hence, $${\rm exc}(f_\lambda) \subseteq \pi(X^{ss}(\lambda_0) \cap X^{us}(\lambda))=\mathbb{B}(L_\lambda),$$ where $\mathbb{B}(L_\lambda) \subseteq Y$ is the 
stable base locus of the line bundle $L_\lambda$ on $Y$, and, similarly, $${\rm exc}(f_\lambda') \subseteq \pi(X^{ss}(\lambda_0) \cap X^{us}(\lambda'))
=\mathbb{B}(L_{\lambda'}).$$ 

We now claim that any extension $f: Y \dasharrow Y_\lambda$ of the rational map $f_\lambda$ to some open subset $O \subseteq Y$ containing 
$Y \setminus \mathbb{B}(L_\lambda)$ contracts (the intersection with $O$ of) every divisorial component of the stable base locus $\mathbb{B}(L_\lambda)$.  
Indeed, by \cite{deb}(Lemma 7.10) (and its proof), there exists a birational morphism $q: \tilde{Y} \to Y$, defining an isomorphism 
outside $q^{-1}(\mathbb{B}(L_\lambda))$, and a birational morphism 
$\tilde{f}: \tilde{Y} \to Y_\lambda$ with $f_\lambda \circ q=\tilde{f}$, such that $\tilde{f}$ contracts all Cartier divisors with support in $q^{-1}(\mathbb{B}(L_\lambda))$. 
Since $Y$ is a geometric quotient, $Y$ is $\Q$-factorial, and hence $\tilde{f}$ in fact contracts all divisors with support in $q^{-1}(\mathbb{B}(L_\lambda))$ that are 
preimages of divisors in $\mathbb{B}(L_\lambda)$.

Since $q$ is a birational morphism, and $Y$ is $\Q$-factorial, the image in $Y$ of the exceptional locus ${\rm exc}(q)$ has codimension at least two (cf. \cite[1.40]{deb}), so that 
$\tilde{f}$ can be identified with a rational map $Y \dasharrow Y_\lambda$, defined on the open subset $Y \setminus q({\rm exc}(q))$, and this rational map 
thus also contracts the divisorial components of $\mathbb{B}(L_\lambda)$. Since any birational extension $f: Y \dasharrow Y_\lambda$ of $f_\lambda$ has to 
agree with $\tilde{f}$ on any open subset where both maps are defined, $f$ also 
contracts the divisorial components of $\mathbb{B}(L_\lambda)$. 

The above argument also applies to $f_\lambda'$, and hence we conclude that all divisorial components in $\mathbb{B}(L_\lambda) \cup \mathbb{B}(L_\lambda')$ lie in the 
exceptional locus ${\rm exc}(f_\lambda)={\rm exc}(f_{\lambda'}) \subseteq \mathbb{B}(L_\lambda) \cap \mathbb{B}(L_\lambda')$. 
Hence, $\mathbb{B}(L_\lambda)$ and $\mathbb{B}(L_{\lambda'})$ coincide in codimension one. Since $\pi: X^{ss}(\lambda_0) \to Y$ defines a geometric 
quotient, this implies that the preimages of $\mathbb{B}(L_\lambda)$ and $\mathbb{B}(L_\lambda')$ coincide in codimension one in $X^{ss}(\lambda_0)$, 
i.e., $$X^{ss}(\lambda) \cap X^{ss}(\lambda_0)=X^{ss}(\lambda') \cap X^{ss}(\lambda_0)$$  in codimension one. 
Finally, since the unstable locus  $X^{ss}(\lambda_0)$ is of codimension at least two, it follows that the identity 
$X^{ss}(\lambda)= X^{ss}(\lambda')$ holds in codimension one. 
\end{proof}

\begin{thm}\label{T: Mori-GITchambers}
 Assume that $\lambda_0 \in C^{\hat{G}}(X)$ is a dominant weight belonging to a $\hat G$-movable chamber, and let $Y:= X^{ss}(\lambda_0)//\hat{G}$ be the corresponding
 quotient. Then the identification $C^{\hat{G}}(X) \cong \overline{\rm Eff}(Y)$ of the $\hat{G}$-ample cone of $X$ with the pseudoeffective cone of $Y$ 
 yields an identification of the GIT-chambers in $C^{\hat{G}}(X)$ with the Mori chambers of $\overline{\rm Eff}(Y)$. 
 
 Moreover, every rational contraction $f: Y \dasharrow Y'$, where $Y'$ is a normal projective variety, is induced by GIT, that is, 
 $Y'=Y_\lambda$, and $f=f_\lambda$, for some $\lambda \in C^{\hat G}(X)$.
\end{thm}

\begin{proof}
Assuming that the strictly dominant weights $\lambda$ and $\lambda'$ are GIT-equivalent, i.e., $X_{\hat G}^{ss}(\lambda)=X_{\hat G}^{ss}(\lambda')$, let 
$$
\varphi: Y_\lambda=X_{\hat G}^{ss}(\lambda)//\hat G \to X_{\hat G}^{ss}(\lambda')//\hat G=Y_{\lambda'}
$$
be the induced isomorphism of the quotients. The 
Mori equivalence of $f_\lambda$ and $f_{\lambda'}$ via $\varphi$ then follows readily from the GIT-descriptions of the rational 
maps $f_\lambda$ and $f_{\lambda'}$ (Lemma \ref{L: ratmapgit}). 
 
 Assume now that the line bundles $L_\lambda$ and $L_\lambda'$ on $Y$, for $\lambda, \lambda'$ in the interior of $C^{\hat G}(X)$ are 
 Mori equivalent. Then, we have a commuting diagram 
 \begin{align*}
 \xymatrix{
 Y \ar@{-->}[dr]^{f_{\lambda'}}\ar@{-->}[r]^{f_\lambda} & Y_\lambda \ar[d]^{\varphi} \\
  & Y_{\lambda'},
 } 
 \end{align*}
where $\varphi$ is an isomorphism of varieties. In order to show that $\lambda$ and $\lambda'$ are GIT-equivalent, it suffices to show the inclusion 
$X^{ss}_{\hat G}(\lambda) \subseteq X^{ss}(\lambda')$ since the same argument will yield the reverse inclusion.
Let therefore $x \in X^{ss}(\lambda)$, and put $y:=\pi_\lambda(x), y':=\varphi(y)$. 
The description of $Y_{\lambda'}$ as the quotient $Y_{\lambda'}=X^{ss}(\lambda')//\hat G$, and the 
fact that the line bundle $\mathcal{L}_{\lambda}$ on $X$ descends to an ample line bundle $A'$ on $Y_{\lambda'}$ shows that there exists a $\hat G$-invariant section 
$\widetilde{s'} \in H^0(X, \mathcal{L}_{\lambda'})^{\hat G}$ and a section $s' \in H^0(Y_{\lambda'}, A')$ with $\widetilde{s'}\mid_{X^{ss}_{\hat G}(\lambda') }=\pi_{\lambda'}^*s'$, and 
$s'(y') \neq 0$. Moreover, $A:=\varphi^*A'$ is an ample line bundle on $Y_\lambda$, and the commutativity of the above diagram shows that the identity of line bundles 
\begin{align*}
 f_\lambda^*A=f_{\lambda'}^*A'
\end{align*}
holds on the open subset $Y \cap \pi(X^{ss}_{\hat G}(\lambda) \cap X^{ss}_{\hat G}(\lambda'))$ of $V$. 
Hence, we have the identity of line bundles 
\begin{align}
 \pi_\lambda^*A=\mathcal{L}_{\lambda'} \label{E: idlbdcodim1} 
\end{align}
on the open subset $O:=X^{ss}_{\hat G}(\lambda_0) \cap X^{ss}_{\hat G}(\lambda) \cap X^{ss}_{\hat G}(\lambda')$ of $X^{ss}_{\hat G}(\lambda)$ (cf. Lemma \ref{L: ratmapgit}). 
Now, by Lemma \ref{L: sscodim1}, $X^{ss}_{\hat G}(\lambda)=X^{ss}_{\hat G}(\lambda')$ in codimension one, so the open subset $O \subseteq X^{ss}_{\hat G}(\lambda)$ has a 
complement of codimension at least two in $X^{ss}_{\hat G}(\lambda)$. Hence, the identity of line bundles \eqref{E: idlbdcodim1} holds on all of $X^{ss}_{\hat G}(\lambda)$. 
In particular, the restriction of the section $\widetilde{s'}$ to $X^{ss}_{\hat G}(\lambda)$ defines a section of $\pi_\lambda^*A$ yielding an extension 
of the section $\pi_\lambda^*\varphi^*s'$ (since they coincide on $X^{ss}_{\hat G}(\lambda_0) \cap X^{ss}_{\hat G}(\lambda) \cap X^{ss}_{\hat G}(\lambda')$). Hence,  
$\widetilde{s'}(x)=(\pi_\lambda^*\varphi^*s')(x)=s(y') \neq 0$, i.e., 
$x \in X^{ss}_{\hat G}(\lambda')$. This shows that $X^{ss}_{\hat G}(\lambda) \subseteq X^{ss}_{\hat G}(\lambda')$, and hence we have proved the first claim 
about the identification of Mori chambers with GIT-chambers. 

Since $Y$ is a Mori dream space, the second part concerning rational contractions follows immediately from the identification of $\overline{{\rm Eff}}(Y)$ with 
$C^{\hat G}(X)$ and the characterization (\cite[Thm. 2.3]{hk}) of rational contractions $f: Y \dasharrow Y'$ onto normal projective varieties $Y'$ as precisely the rational 
contractions $f_D: Y \dasharrow {\rm Proj}(R(Y, \mathcal{O}_Y(D)))$, for effective divisors $D$ on $Y$.
\end{proof}

\begin{thm} \label{T: conesineff}
 The quotient $Y=X^{ss}(C)//\hat G$ is a Mori dream space with $\overline{\rm Eff}(Y)=C^{\hat G}(X)$. This identification 
 of cones, together with the identification of Mori chambers with GIT-chambers, yields an identification of\\
 
 \noindent (i) the nef cone, ${\rm Nef}(Y)$, of $Y$ with the closure $\overline{C}$ of the chamber $C$, \\
 
 \noindent (ii) the movable cone, ${\rm Mov}(Y)$, of $Y$ with the $\hat G$-movable cone ${\rm Mov}^{\hat G}(X)$.
\end{thm}

\begin{proof}
 The nef cone, being the closure of a Mori chamber, corresponds to the closure of some GIT-chamber. Since every integral divisor in the chamber $C$ admits a 
 multiple which descends to an ample divisor on $Y$, the chamber $C$ is the unique chamber corresponding to the nef cone. This proves (i). 
 
 For part (ii), if $D$ is integral divisor on $Y$, let $\pi^*D$ denote the extension to $X$ of the pullback of 
 $D$ by the quotient morphism $\pi: X^{ss}(C)//\hat G \to Y$. The stable base locus, $\mathbb{B}(D)$, of $D$ is then given by 
 \begin{align}
  \mathbb{B}(D)=\pi(X^{us}(D) \cap X^{ss}(C)). \label{E: stablebsloc}
 \end{align}
Since the fibres of $\pi$ all have the dimension $\dim \hat G$, and since the unstable locus $X^{us}(C)$ is of codimension at least two, 
the identity \eqref{E: stablebsloc} shows in particular that $\pi^*D$ is $\hat G$-movable if $D$ is movable.  Hence, ${\rm Mov}(Y) \subseteq {\rm Mov}^{\hat G}(X)$.

Conversely, if $E$ is an integral $\hat G$-movable divisor on $X$, which we can without loss of generality assume to descend to a divisor $\pi_*^{\hat G}(E)$ 
on $Y$, such that $\pi^*\pi_*^{\hat G}(E)=E$, the identity \eqref{E: stablebsloc} applied to $D:=\pi_*^{\hat G}(E)$ shows that $\pi_*^{\hat G}(E)$ is movable. 
Hence we also have the inclusion ${\rm Mov}^{\hat G}(X) \subseteq {\rm Mov}(Y)$. 
\end{proof}

\vspace{0.3cm}

\noindent{\sc Henrik Sepp\"{a}nen, Valdemar V. Tsanov\\

\noindent Mathematisches Institut,
Georg-August-Universit\"at G\"ottingen,\\
Bunsenstra\ss e 3-5, 
D-37073 G\"ottingen,
Germany.}

\noindent{\it Emails:}

\noindent\verb"Henrik.Seppaenen@mathematik.uni-goettingen.de"

\noindent\verb"Valdemar.Tsanov@mathematik.uni-goettingen.de"

\end{document}